\documentclass[a4paper,11pt,twoside,leqno]{article}
\usepackage{cmap}
\usepackage[T1]{fontenc}
\usepackage[latin1]{inputenc}
\usepackage{amsmath,amssymb,amsfonts,amsthm,graphicx,fancyhdr}
\usepackage{rotating}
\usepackage[colorlinks=true]{hyperref}
\usepackage{thmtools}

\newtheorem{teo}{Theorem}[section]
\newtheorem{lem}[teo]{Lemma}
\newtheorem{prop}[teo]{Proposition}
\newtheorem{cor}[teo]{Corollary}
\newtheorem{dfn}[teo]{Definition}
\newtheorem{ques}[teo]{Question}
\newtheorem{conj}[teo]{Conjecture}

\declaretheoremstyle[
  spaceabove=\topsep, spacebelow=\topsep,
  headfont=\bf,  
  notefont=\mdseries, notebraces={(}{)},
  bodyfont=\rmfamily, 
  postheadspace=1em,
  qed=$\Diamond$
]{drem}
\declaretheorem[style=drem, name=Remark, numberlike=teo]{rmk}
\declaretheorem[style=drem, name=Example, numberlike=teo]{exa}

\newcommand{\eg}[0]{\emph{e.g.} }
\newcommand{\ie}[0]{\emph{i.e.} }

\newcommand{\srl}[1]{\overline{#1}}

\DeclareFontFamily{T1}{mafra}{}
\DeclareFontShape{T1}{mafra}{m}{n}{<->s*[0.95]yswab}{} 
\DeclareFontShape{T1}{mafra}{m}{it}{<->s*[1.0]ygoth}{} 
\DeclareTextFontCommand{\textgoth}{\yfrak}

\DeclareSymbolFont{mafrak}{T1}{mafra}{m}{n}
\DeclareSymbolFontAlphabet{\mathfr}{mafrak}

\DeclareSymbolFont{mbbold}{U}{bbold}{m}{n}
\DeclareSymbolFontAlphabet{\mathbbold}{mbbold}

\newcommand{\mr}[1]{\mathrm{#1}}
\newcommand{\mb}[1]{\mathbbold{#1}}
\newcommand{\mc}[1]{\mathcal{#1}}
\newcommand{\ms}[1]{\mathsf{#1}}

\newcommand{\pgen}[1]{\langle #1 \rangle}

\newcommand{\eps}[0]{\varepsilon}

\newcommand{\rr}[0]{\ensuremath{\mathbb{R}}}
\newcommand{\zz}[0]{\ensuremath{\mathbb{Z}}}
\newcommand{\nn}[0]{\ensuremath{\mathbb{N}}}

\newcommand{\qq}[0]{\ensuremath{\mathbb{Q}}}
\newcommand{\cc}[0]{\ensuremath{\mathbb{C}}}

\newcommand{\img}[0]{\mathrm{Im}\,}

\newcommand{\Id}[0]{\mathrm{Id}}

\newcommand{\un}[0]{\mathbbold{1}}

\newcommand{\maxx}[1]{\textrm{\raisebox{.5ex}{\mbox{$\underset{#1}{\max}$}}} \:}

\newcommand{\supp}[1]{\textrm{\raisebox{.5ex}{\mbox{$\underset{#1}{\sup}$}}} \:}

\newcommand{\vide}[0]{\varnothing}

\newcommand{\comp}[0]{\mathsf{c}}

\newcommand{\eqtag}[0]{\addtocounter{teo}{1} \tag{\theteo}}

\newcommand{\cayl}[0]{\mathrm{Cay}_{\text{l}}}
\newcommand{\cayr}[0]{\mathrm{Cay}_{\text{r}}}

\begin{document}

\newcommand{\ud}{\tfrac{1}{2}}
\newcommand{\ut}{\tfrac{1}{3}}
\newcommand{\uq}{\tfrac{1}{4}}

\newcommand{\lmd}{\lambda}
\newcommand{\Lmd}{\Lambda}
\newcommand{\Gm}{\Gamma}
\newcommand{\gm}{\gamma}
\newcommand{\GM}{\Gamma}

\newcommand{\bsl}\backslash
\newcommand{\acts}{\curvearrowright}
\newcommand{\donc}{\rightsquigarrow}

\newcommand{\smdd}[4]{\big( \begin{smallmatrix}#1 & #2 \\ #3 & #4\end{smallmatrix} \big)}

\newcommand{\sgn}{\textrm{sgn}\,}
\newcommand{\cpc}{\mathrm{icp}}

\begin{center}
\Large Mixing, malnormal subgroups and cohomology in degree one.
\vspace*{1cm}

\centerline{\large Antoine Gournay\footnote{Supported by the ERC-StG 277728 ``GeomAnGroup''.}} 
\end{center}

\vspace*{1cm}

\centerline{\textsc{Abstract}}

\begin{center}
\parbox{10cm}{{ \small 
The aim of the current paper is to explore the implications on the group $G$ of the non-vanishing of the cohomology in degree one of one of its representation $\pi$, given some mixing conditions on $\pi$. 
For example, harmonic cocycles of weakly mixing unitary representations factorise by the FC-centre.
In that case non-vanishing implies the FC-centre is trivial or fixes a vector.
Next, for any subgroup $H<G$, $H$ will either be ``small'', almost-malnormal or $\pi_{|H}$ also has non-trivial cohomology in degree one (in this statement, ``small'', reduced \emph{vs} unreduced cohomology and unitary \emph{vs} generic depend on the mixing condition).
The notion of q-normal subgroups is an important ingredient of the proof and results on the vanishing of the reduced $\ell^p$-cohomology in degree one are obtained as an intermediate step.
\hspace*{.1ex} 
}}
\end{center}

\section{Introduction} 

\newcommand{\hdc}[0]{t\!\mc{H}\!\ms{D}^c}
\newcommand{\hdp}[0]{t\!\mc{H}\!\ms{D}^p}
\newcommand{\D}[0]{\ms{D}}
\renewcommand{\L}[0]{\ms{L}}
\newcommand{\Hi}[0]{\ms{H}}
\newcommand{\Uni}[0]{\ms{U}}
\newcommand{\XX}[0]{\mc{X}}
\newcommand{\BB}[0]{\ms{B}}
\newcommand{\Stab}[0]{\mathrm{Stab}}
\newcommand{\WAP}[0]{\mr{WAP}}
\newcommand{\LI}[0]{\mr{Is}_{\text{lin}}}

Representation theory, in particular through cohomology, plays nowadays a central r\^ole in geometric theory as attested by works such as \cite{BHV}, \cite{Oza} and \cite{Shal}. 
The subject matter of this particular work is to try to deduce from the non-vanishing of [reduced or unreduced] cohomology, some algebraic properties which have a taste of hyperbolicity, \eg small centre and large Abelian subgroups are malnormal.

Let us briefly recall that for a linear representation $\pi: G \to \mr{GL}(\ms{V})$ of $G$, a cocycle is a map $b:G \to \ms{V}$ such that the cocycle relation is satisfied: $b(gh) = b(g) + \pi(g)b(h)$. 
The obvious cocycles (those of the form $b(g) = \pi(g) \xi - \xi$ for some $\xi \in \ms{V}$) are called coboundaries, and cohomology is obtained as the usual quotient. 
When $\ms{V}$ and $G$ come with a topology, one can speak of reduced cohomology by looking at the closure of the space of coboundaries for the topology of uniform convergence on compacts.

For unitary representation, Guichardet \cite{Gu} was the first to notice that the reduced cohomology can be identified to a natural subspace of cocycles, the harmonic cocycles. 
If $G$ is generated by a finite set $S$ and $\mu$ is some probability measure on $S$ such that $\mu(s^{-1}) = \mu(s)$, a $\mu$-harmonic cocycle is a cocycle such that $\sum_{s \in S} \mu(s) b(s) =0$.
See Bekka \cite{Bek}, \cite{GJ}, Erschler \& Ozawa \cite{EO} and Ozawa \cite{Oza} for recent works where these cocycles play an important part; see Bekka \& Valette \cite{BV} for prior results which show their importance.

Recall that the kernel of a cocycle $\ker b = b^{-1}(0)$ is a subgroup of $G$ and that the FC-centre of $G$, noted $Z^{FC}(G)$, is the characteristic subgroup of elements with a finite conjugacy class ($Z^{FC}(G)$ contains the centre, $Z(G)$).
The first result (essentially a concatenation of Theorem \ref{tergliou-t} and Theorem \ref{twmconj-t}) shows that non-trivial reduced cohomology leads to the FC-centre fixing a vector or being trivial.
\begin{teo}\label{tcen-t}
Let $G$ be a finitely generated group and $\pi$ a unitary representation.  \\
\hspace*{1ex}{\bf 1.} If $\pi$ is weakly mixing, then, for any harmonic cocycle $b$, $Z^{FC}(G) \subset \ker b \cap \ker \pi_{| \img b}$. \\
\hspace*{1ex}{\bf 2.} If $\pi$ is ergodic and $G$ is $\mu$-Liouville (for some symmetric measure $\mu$ of finite support), then, for any harmonic cocycle $b$, $Z(G) \subset \ker b \cap \ker \pi_{|\img b}$.\\
\emph{In particular,} \\
 \hspace*{1ex}{\bf 3.} if $\pi$ is weakly mixing and there is a non-trivial subgroup $H$ of $Z^{FC}(G)$ such that $\pi_{| H}$ is ergodic, then $\srl{H}^1(G,\pi)=0$.\\
 \hspace*{1ex}{\bf 4.}
 if $\pi$ is ergodic, $G$ is $\mu$-Liouville and there is a non-trivial subgroup $H$ of $Z(G)$ such that $\pi_{| H}$ is ergodic, then $\srl{H}^1(G,\pi)=0$.
\end{teo}
The proof of this Theorem may be found at the end of \S{}\ref{sscentraliser}.
Corollary \ref{tmimicentre-c} gives a very similar result.
These results can be used to give a proof that nilpotent groups have property $H_T$ and virtually nilpotent groups have property $H_{FD}$ (in the sense of Shalom \cite{Shal}). See Corollary \ref{tguich-c} and Corollary \ref{twmvnil-c} (these statements go back to Guichardet \cite[Th\'eor\`eme 7 in \S{}8]{Gu}). It can be also used to show that in an amenable group which is torsion either no harmonic cocycle is proper or the FC-centre is finite (see Corollary \ref{tcoproptor-c}). 

Theorem \ref{tcen-t} and Corollary \ref{tmimicentre-c} are improvements of \cite[Proposition 1.5 and Lemma 2.7]{GJ}.
See also Bekka, Pillon \& Valette \cite[\S{}4.6 and Corollary 4.13]{BPV} for similar results.  
Recently, Brieussel \& Zheng \cite[\S{}3]{BZ} gave an application of Theorem \ref{tcen-t} to a question of Shalom.

An important weakening of the normality relation (q-normal and wq-normal subgroups, see \S{}\ref{ssqnorm}) was introduced by Popa \cite{Po} while studying cohomology.
Peterson \& Thom \cite[Remark 5.3]{PT} showed that these notions are closely related to the presence of almost-malnormal (even malnormal if the group is torsion-free) subgroups.
They further used this to show that many groups have trivial reduced $\ell^2$-cohomology.
The second main result is to show that if $G$ has non-vanishing of cohomology for a representation $\pi$, then the subgroups of $G$ satisfy a trichotomy: they are either ``small'', almost-malnormal or the restricted representation has non-vanishing cohomology.

The precise formulation is a concatenation of Theorem \ref{tteoptweak-t}, Corollary \ref{tcoefflpharm-c},  Theorem \ref{tqnorlpcoh-t} and Theorem \ref{tqnorcen-t}. 
A representation is said to have finite stabilisers, if the stabiliser of any $\xi \in \ms{V} \setminus \{0\}$ is finite (this condition is weaker than mildly mixing, stronger than ergodic and not comparable to weakly mixing).
The FC-centraliser $Z^{FC}_G(H)$ of a subgroup $H<G$ is defined in \S{}\ref{sscentraliser} (it contains the centraliser).
\begin{teo}\label{tmaln-t}
Let $G$ be a finitely generated group.
\begin{enumerate}\renewcommand{\itemsep}{-1ex} \renewcommand{\labelenumi}{\bf \arabic{enumi}.}
\item If 
$\pi$ is a linear representation with finite stabilisers and
there is a subgroup $H<G$ with $H^1(H,\pi_{|H}) =0$, \\
then either $H$ is contained in an almost-malnormal strict subgroup of $G$, $H$ is finite, or $H^1(G,\pi)=0$.

\item If $\pi$ is a unitary representation with finitarily coefficients in $\ell^q$ and
there is a finitely generated subgroup $H < G$ with $\srl{H}^1(H,\lmd_{\ell^pH}) =0$ (where $p>q$) or $Z^{FC}_G(H)$ infinite, \\
then either $H$ is contained in an almost-malnormal strict subgroup of $G$, $H$ has growth bounded by a polynomial of degree $d<p$, or $\srl{H}^1(G,\pi)=0$.
\end{enumerate}
\end{teo}
Recall that, if $G$ is torsion-free, almost-malnormal subgroups are actually malnormal.

In Theorem \ref{tmaln-t}.1, the ``smallness'' of $H$ also has an equivocal formulation: since trivial unreduced cohomology implies that the space of coboundaries is closed, one might be tempted to say that the group is large for the representation under consideration.
Though Theorem \ref{tmaln-t}.1 only uses the work of Popa \cite[\S{}2]{Po} and Peterson \& Thom \cite[\S{}5]{PT}, Theorem \ref{tmaln-t}.2 also uses harmonic cocycles. 

The first (and shortest) step in the proof of Theorem \ref{tmaln-t}.2 (Corollary \ref{tcoefflpharm-c}) is to show that harmonic cocycles for such representations gives rise to a harmonic function with a gradient in $\ell^p$ which in turn implies that the reduced $\ell^p$-cohomology in degree one is non-trivial (by a result of \cite[Theorem 1.2 or Corollary 3.14]{Go}).
The second step is to extend the groups for which the $\ell^p$-cohomology vanishes (see Remark \ref{rreflpcoh} for references).
The trichotomy expressed in reduced $\ell^p$-cohomology is a combination of Theorem \ref{tqnorlpcoh-t} and Theorem \ref{tqnorcen-t}:
\begin{cor}\label{tcortrico-c}
Assume $\srl{H}^1(G,\lmd_{\ell^p G}) \neq 0$ and $K<G$ are finitely generated. Then at least one of the following holds:
\begin{itemize}\renewcommand{\itemsep}{-1ex}
 \item $\srl{H}^1(K,\lmd_{\ell^p K}) \neq 0$ and $Z^{FC}_G(K)$ is finite,
 \item $K$ has growth at most polynomial of degree $d \leq p$,
 \item $K$ is contained in a almost-malnormal strict subgroup of $G$.
\end{itemize}
\end{cor}
Corollary \ref{tcorptlp-c} also gives a similar statement for unreduced cohomology which does not require finite generation (it follows from Theorem \ref{tmaln-t}.1 and a result of Martin \& Valette \cite[Proposition 2.6]{MV}).

Corollary \ref{tcortrico-c} may also be seen as an algebraic version of the following geometric result from \cite[\S{}4.2 and Corollary 4.1]{Go}. If $G$ is finitely generated, $\srl{H}^1(G,\lmd_{\ell^p G}) \neq 0$ and $\Gamma$ is a spanning subgraph of some Cayley graph of $G$ then at least one of the following holds:
\begin{itemize}\renewcommand{\itemsep}{-1ex}
 \item the reduced $\ell^p$-cohomology of $\Gamma$ is non-trivial,
 \item $\Gamma$ may not have $d$-dimensional isoperimetry for any $d > 2p$,
\end{itemize}
Theorem \ref{tqnorlpcoh-t} can be used to show the triviality of reduced $\ell^p$ cohomology in degree one for all $p \in ]1,\infty[$ (and, as a consequence of Corollary \ref{tcoefflpharm-c}, the triviality for any unitary representation with finitarily coefficients in $\ell^p$) for non-solvable Baumslag-Solitar groups (see Example \ref{exBS}), many solvable groups (see Example \ref{exFM} and Corollary \ref{tsolvlp-c}) or even hyperabelian groups (see Corollary \ref{tasclp-c}).

On the way, it is also shown that the $p$-Royden boundary of many groups consists in only one point (see Corollary \ref{tcorroyd-c}).

{\it Organisation of the paper.} 
\S{}\ref{smix} is mostly concerned with definitions (the various conditions of mixing, coefficients and WAP functions are introduced in \S{}\ref{ssmixmean}; the various notions of normality and their relations are discussed in \S{}\ref{ssqnorm}) and the proof of Theorem \ref{tmaln-t}.1 (\ie  Theorem \ref{tteoptweak-t}) as it follows from fairly generic considerations. 
\S{}\ref{sredcoh} splits as follows: the notion of reduced cohomology and harmonic cocycles are done in \S{}\ref{ssharmcoc} (and generalised to Banach space in \ref{ssbanachrep}), further mixing conditions (finitarily coefficients in $\ell^p$) as well as their implications on harmonic functions are discussed in \S{}\ref{ssgrad}, and the fact that cocycles are ``virtual coboundary'' is dealt with in \S{}\ref{ssvirtcob}. 
\S{}\ref{scen} contains the proof of Theorem \ref{tcen-t}; it relies on \S{}\ref{ssmixmean}, \S{}\ref{ssharmcoc} and (tangentially) \S{}\ref{ssgrad}.
\S{}\ref{slpcoh} contains the proof of Theorem \ref{tmaln-t} and Corollary \ref{tcortrico-c}; it relies on \S{}\ref{ssqnorm} and (most of) \S{}\ref{sredcoh}. 
\S{}\ref{sslprem} gives a very short introduction to $\ell^p$-cohomology and its application to unitary representation, 
\S{}\ref{ssbndval} is dedicated to proving the criterion of triviality in terms of being constant at infinity (and addresses some questions of $p$-parabolicity), 
\S{}\ref{sslpqnor} contains the proof of Corollary \ref{tcortrico-c} 
and \S{}\ref{sslpcor} investigates further examples and corollaries.
\S{}\ref{sques} contains questions and open problems raised by the current investigation.

{\it Acknowledgments:} The author would like to thank: 
A.~Thom, for indicating that q-normality is a very useful tool to prove triviality of cocycles and pointing to \cite[\S{}5]{PT};
M.~Bourdon, for many interesting comments and explaining to the author that relatively hyperbolic groups have a non-trivial $\ell^p$-cohomology for some $p\geq 1$;
A.~Carderi, for pointing out the existence of intermediate notions of mixing (mildly mixing) and for explaining how to extend the proof of Theorem \ref{tvirtcoc-t} to the case of a group acting by measure preserving transformations on $(X,\mu)$;
P.N.~Jolissaint, since part of the present work is a natural prolongation of a question we investigated in \cite{GJ};
D.~Osin, for pointing out that some acylindrically hyperbolic groups may have trivial $\ell^p$-cohomology for all $p \in [1,\infty[$;
M.~Schmidt, for explaining the note on uniform transience \cite{KLSW} thus enabling to make a much more elegant proof of the ``trivial in reduced $\ell^p$-cohomology if and only if one value at infinity'' from \cite[\S{}4.2 and Corollary 4.1]{Go} (see Corollary \ref{tbndval-c}); 
and 
A.~Valette, for various comments and references.

The author would also like to thank Y.~Cornulier as well as D.~Holt and P.~de la Harpe for pointing out to litterature on the FC-centre and malnormal subgroups (respectively) via the website MathOverflow.

\par This work is supported by the ERC-StG 277728 ``GeomAnGroup'' .

\section{Mixing, q-normal subgroups and untwisting}\label{smix}

\subsection{Mixing and means}\label{ssmixmean}

Standard references for this subsection are the books of Glasner \cite{Glasner} (in particular, Ch.1 \S{}9-10, Ch.3 \S{}1-2 and Ch.8 \S{}5) and Kechris \cite{Kechris} (in particular, Ch.I \S{}2, Ch.II \S{}10 and App. H).

Before introducing the mixing conditions that are relevant for the current purposes, let us briefly recall some things about WAP functions.
Recall that a group $G$ acts on any function $f:G \to \cc$ by translation: $\gm \cdot f(x) = f(\gm^{-1} x)$. 
Upon restriction to $\ell^p G$ (for any $p \in [1,\infty]$) this action is isometric. 
It also preserves $c_0G$ (the closure in $\ell^\infty$-norm of finitely supported functions).
The space of \textbf{WAP functions} a subspace of $\ell^\infty G$ defined by 
\[
\mr{WAP}(G) = \{f \in \ell^\infty G \mid \srl{G \cdot f}^{\text{weak}} \text{ is compact [in the weak topology]}\}
\]
where $\srl{G \cdot f}^{\text{weak}}$ is the weak closure of the $G$-orbit of $f$.

For $\mc{X}$ a closed subspace of $\ell^\infty G$, a \textbf{mean} on $\mc{X}$ is an element $m \in (\ell^\infty G)^*$ such that $m(\un_G) = 1$ (where $\un_G$ is the function taking constant value $1$) and $m(f) \geq 0$ whenever $f \geq 0$.
Using the Ryll-Nardzevsky theorem, one gets the surprising fact that there is a unique $G$-invariant mean on $\WAP(G)$ ($G$-invariance is with respect to the afore-mentioned action of $G$ on functions). 
In particular, this implies that, for an amenable group, all the invariant means on $\ell^\infty G$ coincide on $\WAP(G)$.

Given an isometric linear representation of $G$ on a Banach space $\BB$ (\ie a homomorphism $\pi:G \to \LI(\BB)$), 
the \textbf{coefficient} of $\eta \in \BB^*$ and $\xi \in \BB$ is the function $\kappa_{\eta,\xi}: G \to \cc$ defined by $\kappa_{\eta,\xi}(x) = \pgen{ \eta \mid \pi(x) \xi }$.
When $\BB$ is a Hilbert space, the distinction between $\BB$ and $\BB^*$ is usually dropped (\ie $\eta$ and $\xi$ are in $\BB$).
It turns out that coefficient of representation on reflexive Banach space are in $\WAP(G)$.

\begin{dfn}\label{defsynd}
A set $L \subset G$ is syndetic if there is a finite $S \subset G$ to that $SL = G$. 
$L$ is thick if for all finite $F \subset G$, $\cap_{f \in F} fL \neq \vide$.
$L$ is \textbf{thickly syndetic}, if for any finite $F \subset G$, $L' = \cap_{g \in F} gL$ is syndetic.
 
A function is a \textbf{flight function} iff for any $\eps>0$, $D_\eps := \{ x \in G \mid |f(x)|\leq \eps\}$ is thickly syndetic.
\end{dfn}

A subgroup is syndetic if and only if it has finite index. A subgroup cannot be thick (or thickly syndetic) unless it is the whole group.

In the upcoming definitions, $m$ will denote the unique invariant mean on $\WAP(G)$. The notation $\underset{g \to \infty}\liminf F(g)$ should be understood as $\underset{n \to \infty}\lim \inf \{ F(g) \mid g \notin B_n\}$ where $B_n$ is some sequence of increasing finite sets with $\cup B_n = G$. The space $c_0G$ is the closure of finitely supported functions in $\ell^\infty G$, \ie functions $F$ with $\underset{g \to \infty}\limsup |F(g)| = 0$.
\begin{dfn}\label{defmix}
Let $\pi: G \to \LI(\BB)$ be a linear isometric representation of $G$. $\pi$ is 
\begin{enumerate}\renewcommand{\itemsep}{-1ex} \renewcommand{\labelenumi}{\bf \arabic{enumi}.}
\item \textbf{ergodic} if 
its coefficients $\kappa_{\eta,\xi}$ have mean $0$, \ie $\forall \eta \in \BB^*$ and $ \forall \xi \in \BB$, one has $m(\kappa_{\eta,\xi}) =0$.
\item \textbf{weakly mixing} if
$\forall \eta \in \BB^*, \forall \xi \in \BB,$ the coefficient $\kappa_{\eta,\xi}$ is a ``flight function'', \ie $m(|\kappa_{\eta,\xi}|) =0$.
\item \textbf{mildly mixing} if
$\forall \xi \in \BB,$ $\underset{g \to \infty}\liminf \| \pi(g) \xi - \xi\| >0$.
\item \textbf{strongly mixing} if 
$\forall \eta \in \BB^*, \forall \xi \in \BB,$ the coefficient $\kappa_{\eta,\xi}$ belongs to $c_0G$.
\end{enumerate}
\end{dfn}
Note that ergodicity is equivalent to the fact that $\pi$ does not contain the trivial representation, \ie there are no invariant vectors.
For unitary representations, weakly  mixing is equivalent to the fact that $\pi$ does not contain a non-zero finite-dimensional subrepresentation.
Mildly mixing implies that, upon restriction to any infinite subgroup, the action is ergodic.

These definitions are in a monotone order: strongly mixing implies mildly mixing implies weakly mixing implies ergodic. 
For unitary representations, it is known that these implications are not reversible as soon as the group has an element of infinite order.

On some rare occasions, other (possibly inequivalent) definition of mixing will be used (this will be stressed). 
These definitions comes from action of groups by measure preserving transformations.
\begin{dfn}\label{defmixact}
Let $(X,\mu)$ be a measure space. 
Assume $G \acts (X,\mu)$ by measure preserving transformations. The action is
\begin{enumerate}\renewcommand{\itemsep}{-1ex} \renewcommand{\labelenumi}{\bf \arabic{enumi}.}
\item is ergodic if for any $G$-invariant subset $A \subset X$, either $\mu(A)=0$ or $\mu(A^\comp) = 0$.
\item is mildly mixing if $\forall A \subset X$ with $\mu(A) \notin \{0,\mu(X)\}$ one has $\liminf_{ g \to \infty } \mu( g A \triangle A)  >0$.
\item is strongly mixing if $\forall A \subset X$, $\lim_{ g \to \infty } \mu( g A \cap B) = \mu(A) \mu(B)$.
\end{enumerate}
\end{dfn}
Given such an action, one can associate a isometric (resp. unitary) representation by looking at $\L^p(X,\mu)$ (resp. $\L^2(X,\mu)$) and letting $G$ act on $f \in \L^p(X,\mu)$ by $\gm f(x) = f(\gm^{-1}x)$. 
If $\mu(X) < \infty$, one normally considers restricts the action to the subspace $\L^p_0(X,\mu) = \{ f \in \L^p(X,\mu) \mid \int_X f \mr{d} \mu=0\}$ (to avoid the obvious invariant vector given by the constant function).
When $\mu(X) = \infty$, the convention in the present work is that $\L^p_0(X,\mu) = \L^p(X,\mu)$.
These representations will here be called $\L^p$\textbf{-representations}.
For $p=2$, this representation (called a Koopman representation) has the same mixing properties as the action.

As an example, let us mention what those properties become for a countable $X$ (with the counting measure).
\begin{dfn}
Let $X$ be countable. The action $G \acts X$ 
\begin{enumerate}\renewcommand{\itemsep}{-1ex} \renewcommand{\labelenumi}{\bf \arabic{enumi}.}
\item is ergodic if there are no fixed points.
\item is weakly mixing if all $G$-orbits are infinite.
\item is mildly mixing [equivalently, strongly mixing] if there are no infinite stabilisers.
\end{enumerate}
\end{dfn}

\subsection{Cocycles and q-normality}\label{ssqnorm}

A nice reference on cocycles of representations is Bekka, de la Harpe \& Valette \cite[Chapter 2]{BHV} contains a complete discussion.
\begin{dfn}
Let $\pi$ be a linear isometric representation of $G$ on a Banach space $\BB$. 
A \textbf{cocycle} with values in $\pi$ is a map $b : G \rightarrow \BB$ satisfying the [cocycle] relation
\[
\forall g,h\in G, \qquad 
b(gh)=\pi(g)b(h)+b(g)
 \]
A cocycle of the form $g \mapsto (\pi(g)-1)v$, for some $\xi \in \BB$, is called a \textbf{coboundary}. 

$Z^{1}(G,\pi)$ (resp. $B^{1}(G,\pi)$) denotes the space of cocycles (resp. coboundaries) with respect to $\pi$. 
The first cohomology space of $\pi$ is defined as the quotient space  
\[
H^{1}(G,\pi)=Z^{1}(G,\pi)/B(G,\pi).
\] 
\end{dfn}
For later use, let us introduce the \textbf{coboundary map} $d : \BB \to Z^{1}(G,\pi)$ defined by $(d\xi)(g)=(\pi(g)-1)\xi$, for any $g\in G$ and $\xi\in\BB$. 
The space of \textbf{coboundaries}, denoted $B^1(G,\pi)$, corresponds to the image of $\BB$ via $d$.

Note that, a map $b$ satisfies the cocycle relation, if and only if, $G$ acts by affine isometries on $\BB$ by $\alpha(g)v=\pi(g)v+b(g)$. 
Such an affine action $\alpha$ has a fixed point if and only if $b$ is a coboundary.
Furthermore, for real Banach space, the Mazur-Ulam theorem (see Nica \cite{Nica-MU} for a very nice proof) 
can be used to show that any isometric representation $\sigma$ of $G$ is given by a cocycle $b$ of a linear isometric representation $\pi$ by the rule $\sigma(g)v = \pi(g)v + b(g)$ (for any $v \in \BB$).

Here are some simple consequences of the cocycle relation. First, $b(e) =0$ (since $b(e) = b(e \cdot e) = b(e) + \pi(e) b(e) = 2b(e)$). 
From there one gets
\[
b(g^{-1}) = -\pi(g^{-1}) b(g),
\]
by applying the cocycle relation to $b(g^{-1}g)$. 
This can be used to get
\[\eqtag \label{eqcocconj}
b(ghg^{-1}) = (1-\pi(ghg^{-1}) ) b(g) + \pi(g) b(h)
\]
and then rewritten (with $k= ghg^{-1}$) as
\[\eqtag \label{eqcocpopa}
b(g) - \pi(k) b(g) = b(k) - \pi(g) b(g^{-1}kg).
\]
\begin{rmk}\label{ractnorm}
Recall that when $H \lhd G$, there is an action of $G$ on $Z^1(H,\pi_{|H})$ by $g \cdot b(h) = \pi(g) b(g^{-1}hg)$. 
Furthermore (using \eqref{eqcocpopa}), this actions leaves $B^1(H,\pi_{|H})$ invariant (so passes to $H^1(H,\pi_{|H})$) and $H$ acts trivially on $H^1(H,\pi_{|H})$. Lastly, if $\pi_{|H}$ is ergodic, then $H^1(G,\pi) \simeq H^1(H,\pi_{|H})^{G/H}$.
\end{rmk}
So if $H \lhd G$, $\pi_{|H}$ is ergodic and $H^1(H,\pi) =0$ then $H^1(G,\pi)=0$. 
In the subsequent subsections, other hypothesis which allow to deduce the triviality of the cohomology of $G$ from that of $H$ will be given. 
To this end, let us recall some variations of the notions of normality (see Peterson \& Thom \cite[\S{}5]{PT} and Popa \cite[\S{}2]{Po}). 
\begin{itemize}\renewcommand{\itemsep}{-1ex}
\item A subgroup $H < G$ is called \textbf{q-normal} if there is a generating set $A$ of $G$ such that, $\forall g \in A$, $g H g^{-1} \cap H$ is infinite.
\item A subgroup is \textbf{ascendant}\footnote{Also called ``w-normal'', but also sometimes called descendant.} [resp. \textbf{wq-normal}] if there exists an ordinal number $\alpha$, and an increasing chain of subgroups, such that $H_0 = H$, $H_\alpha = G$, and for any $\gamma \leq \alpha$, $\cup_{\beta < \gamma} H_\beta < H_\gamma$ is normal [resp. q-normal].
\item A subgroup $K< G$ is \textbf{malnormal} if $K \cap g K g^{-1} = \{e\}$ for any $g \in G \setminus K$.
\item A subgroup $K< G$ is \textbf{almost-malnormal} if $K \cap g K g^{-1}$ is finite for any $g \in G \setminus K$.
\end{itemize}
According to the previous definition, $G$ is a malnormal and an almost-malnormal subgroup of $G$.
Since $K \cap gKg^{-1}$ is also a subgroup of $G$, the notion of almost-malnormal and malnormal subgroups are equivalent in torsion-free groups.

Recall that the intersection of two almost-malnormal subgroups is almost-malnormal. 
The \textbf{almost-malnormal hull} of a subgroup $H < G$ is the intersection of all almost-malnormal subgroups\footnote{When $H$ is finite, it is its own almost-malnormal hull.} containing $H$.

Obviously, a q-normal subgroup is always infinite.
The generating set (of $G$) in the definition of q-normal subgroup may always be assumed symmetric. 
Indeed ``$K \cap g Kg^{-1}$ is infinite'' implies that $g^{-1} ( K \cap gKg^{-1})g = g^{-1}Kg \cap K$ is infinite.

The sequences $\{H_\beta\}$ coming in the definition of ascendant (resp. wq-normal) subgroups are called \textbf{ascending} normal (resp. q-normal) sequences.

\begin{dfn}
Let $K < G$ be an infinite subgroup. 
The \textbf{q-normaliser} of $K$ in $G$ is the largest subgroup in which $K$ is q-normal:
\[
N^q_G(K) := \pgen{ g \in G \mid gKg^{-1} \cap K \text{ is infinite} }
\]
\end{dfn}
The q-normalisers are in some respect better behaved than normalisers. 
Recall that hypernormalising groups (a group where any subgroup $H$ which is ascendant with respect to a finite normal sequence sees its sequence of iterated normalisers converge to the full group) are rather rare; for finite groups see Heineken \cite{Heineken}.

Let us define the transfinite sequence of iterated q-normalisers $\{ N^{q,\beta}_G(H) \}$ by $N^{q,0}_G(H) = H$ and $N^{q,\alpha}_G(H) = N^q_G\big( \cup_{\gamma < \alpha} N^{q,\gamma}_G(H) \big)$.
Say the sequence $\{N^{q,\alpha}_G(H)\}$ stabilises if there is a $\alpha$ so that $\forall \beta \geq \alpha$, $N^{q,\beta}_G(H) = \cup_{\gamma < \beta} N^{q,\gamma}_G(H)$).
\begin{lem}\label{twqnor-amal}
Let $H < G$ be an infinite subgroup
\begin{enumerate}\renewcommand{\itemsep}{-1ex} \renewcommand{\labelenumi}{\bf \arabic{enumi}.}
 \item If $K<G$ is another subgroup, then $H<K \implies N^q_G(H) < N^q_G(K)$.
 \item If $M<G$ is an almost-malnormal subgroup, then $H<M \implies N^q_G(H) < M$.
 \item The sequence $\{ N^{q,\beta}_G(K) \}$ stabilises at the subgroup $M$ which is the almost-malnormal hull of $H$.
 \item $H$ is wq-normal if and only if there is an ordinal $\alpha$ with $N^{q,\alpha}_G(K) = G$.
\end{enumerate}
\end{lem}
Compare with Peterson \& Thom \cite[Lemma 5.2 and Remark 5.3]{PT} for a slightly different argument. 
\begin{proof}
\textbf{1}: is direct from the definitions. Since $g H g^{-1} \cap H \subset g Kg^{-1} \cap K$ and $g$ is so that the left side of the inclusion is infinite, then so is the right side. 

\textbf{2}: is also easy. Again $g H g^{-1} \cap H \subset g Mg^{-1} \cap M$. For any $g \notin M$, right side is finite, hence the left side is also finite.

\textbf{3}: 
Since at least one element is added at the ordinals where the sequence is not stable, the sequence has to stabilise at some point (at the latest when $|\alpha| > |G|$).
Let $K = N^{q,\alpha}_G(H)$ be the subgroup where the sequence stabilises and $M$ be the almost-malnormal hull of $H$.
By definition, if $N^q(K) = K$ then $K$ is almost-malnormal; hence $M < K$.
By 2, no element of $G \setminus M$ can belong to $N^{q,\alpha}_G(H)$ (for any $\alpha$); hence $K < M$

\textbf{4} is a direct consequence of 3.
\end{proof}
In fact, the sequence $N^{q,\alpha}_G(H)$ is actually the shortest ascending q-normal sequence making $H$ wq-normal in its almost-malnormal hull $M$.
Indeed, if $H_0 < H_1$ and $H_1 < H_2$ are q-normal inclusions, then $H_1 < N^q(H_0)$ and $H_2 < N^{q,2}(H_0)$.



In short, by Lemma \ref{twqnor-amal} or Peterson \& Thom's \cite[Remark 5.3]{PT}, every infinite subgroup $H < G$ is wq-normal in an almost-malnormal subgroup. 
The case where $G$ itself is the almost-malnormal subgroup is wq-normality.

Also, if $H<K<G$ and $H$ is q-normal in $G$ then $K$ is also q-normal in $G$. 
However, if $H$ is q-normal in $K$ and $K$ is q-normal in $G$, it could happen that $H$ is not q-normal in $G$.
Another weakening of normality (s-normal) is better behaved in this respect (see Peterson \& Thom \cite[\S{}5]{PT}).

A nice reference on malnormal subgroups in infinite groups is the work of de la Harpe \& Weber \cite{HW}. For example, \cite[Proposition 2.(vii)]{HW} shows that a group without $2$-torsion and with an infinite cyclic normal subgroup has no malnormal subgroups.


\subsection{Untwisting cocycles } \label{suntwist}

The remainder of this section is dedicated to adapt the proof a result of Peterson \& Thom \cite[Theorem 5.6]{PT} which concerns the reduced cohomology of $\lmd_{\ell^2G}$, the left-regular representation on $\ell^2G$ (see also Popa \cite[Lemma 2.4]{Po} for a prior version of these arguments on another type of cocycles). 

A linear representation $\pi: G \to \mr{GL}(\ms{V})$ (where $\ms{V}$ is an infinite dimensional vector space) is said to have \textbf{finite stabilisers} if for all $\xi \in \ms{V} \setminus \{0\}$, the set $\{ g \in G \mid \pi(g) \xi = \xi\}$ is finite.
Finite stabilisers is equivalent to: for any infinite subgroup $H <G$, $\pi_{| H}$ is ergodic.

Any mildly mixing representation has finite stabilisers. 
There are unitary representations with finite stabilisers which are not weakly mixing 
(consider $\pi: \zz \to \Uni(1)$ given by $z \mapsto e^{2 \pi \imath z \alpha}$ where $\alpha \in \rr \setminus \qq$) 
and there are weakly mixing unitary representations which do not have finite stabilisers 
(take any mixing representation on $G$ and extend it trivially to $G \times \zz$; see the end of \S{}\ref{sques} for an interesting example due to Shalom \cite[Theorem 5.4.1]{Shal}). 

Finite stabilisers [in infinite groups] implies that the representation has \textbf{no finite subrepresentations} 
(\ie does not contain a subrepresentation which factors through a finite quotient group) 
and, consequently, is ergodic.
So ``finite stabilisers'' and ``weakly mixing'' are two conditions which lie between ``mildly mixing'' and ``no finite subrepresentations'' but are inequivalent. 
Finite stabilisers is better behaved in that it is inherited by subrepresentations and infinite subgroups.

Given a representation $\pi$ and a cocycle $b \in Z^1(G,\pi)$, let $\ker b = \{ g \in G \mid b(g) =0\}$.
The following lemma is a straightforward adaptation of a part of the proof of Peterson \& Thom's \cite[Theorem 5.6]{PT} or Popa \cite[Lemma 2.4]{Po}.
\begin{lem}\label{tlempt-l}
Assume $\pi$ has finite stabilisers and let $b \in Z^1(G,\pi)$ be a cocycle, then $\ker b$ is a subgroup which is [either finite, equal to $G$, or] almost-malnormal in $G$.
\end{lem}
\begin{proof}
The cocycle relation shows that $K:= \ker b$ is a subgroup: $b(e) =0$ (so $e \in K$), $b(gh) = b(g) + \pi(g) b(h) = 0 + \pi(g) 0 = 0$ and $b(g^{-1}) = - \pi(g^{-1})b(g) =0$.
Assume $K \neq G$, $K$ is infinite and $K$ is not almost-malnormal. 
Then there exists $g \in G \setminus K$ with $gKg^{-1} \cap K$ infinite. 
Now, $\forall k \in gKg^{-1} \cap K$, \eqref{eqcocpopa} gives $b(g) - \pi(k)b(g) = b(k) - \pi(g) b(g^{-1}kg) = 0$.
Since $gKg^{-1} \cap K$ is infinite, $b(g) = \pi(k) b(g)$ for infinitely many $k \in G$. 
However, $\pi$ has finite stabilisers, so this implies that $b(g) = 0$. 
But then $g \in K$, a contradiction. 
\end{proof}
If one is a wee-bit more careful in writing it down, the proof of Lemma \ref{tlempt-l} works even if $(V,+)$ is not Abelian, \ie it still works in the case $(V,+)$ is a [generic] group and $\pi: G \to \mr{Aut}(V)$ a homomorphism.

That said, the proof of Peterson \& Thom \cite[Theorem 5.6]{PT} goes on without problems if there are no almost-invariant vectors. 
\begin{teo}\label{tteoptweak-t}
Assume $\pi$ has finite stabilisers and there is a infinite subgroup $H < G$ so that $H^1(H,\pi_{| H}) = 0$.
Let $K$ be the almost-malnormal hull of $H$.
Then $H^1(K,\pi_{|K}) =0$

In particular, if $H$ is wq-normal in $G$, then $H^1(G,\pi) = 0$.
\end{teo}
\begin{proof}
Assume $b \in H^1(K,\pi_{|K})$ is non-trivial. 
By hypothesis, $b_{| H}$ has trivial class. 
Hence there is a $z \in B^1(H,\pi)$ so that $ b = z$ on $H$. 
Since $z(g) = \pi(g) \xi - \xi$ for some $\xi \in \ms{V}$, it turns out that $z \in B^1(G,\pi)$. 
Let $b' = b-z$, then $\ker b' \supset H$. 
By Lemma \ref{tlempt-l} and since $H$ is infinite, $\ker b'$ must be an almost-malnormal subgroup containing $H$.
This means that $\ker b'$ contains $K$.
Hence $b \equiv z$ on $K$, \ie $[b] = 0 \in H^1(K,\pi_{|K})$.

$H$ is wq-normal in $G$ if and only if almost-malnormal hull of $H$ is $G$; see Lemma \ref{twqnor-amal} or Peterson \& Thom's \cite[Remark 5.3]{PT}. 
\end{proof}
Like Popa \cite[\S{}2]{Po}, it is sometimes useful to put the mixing condition on the subgroup instead. 
For example, 
\begin{lem}\label{tlempopastyle-l}
Let $b \in Z^1(G,\pi)$ be a cocycle and let $H < \ker b$ be a subgroup of its kernel.
If $\pi_{|H}$ has finite stabilisers and $H<G$ is q-normal, then $\ker b=G$.
\end{lem}
The proof is exactly as in Lemma \ref{tlempt-l}.
One can then use transfinite induction to obtain Theorem \ref{tteoptweak-t}. 
However, one then needs to check that $\pi_{|H_\alpha}$ has finite stabilisers for any $H_\alpha$ in the ascending q-normal sequence starting at $H_0 = H$.
So in the end, the hypothesis is clumsier (though weaker) to formulate than ``$\pi$ has finite stabilisers''.

%
%
%
%

\section{Reduced cohomology}\label{sredcoh}

The aim of this section is to introduce reduced cohomology and to show how to reduce its study to that of harmonic (or $p$-harmonic) functions. 
There are two possibilities to introduce harmonic functions. The first is by considering cocycles which are ``minimal'' for some norm (\S{}\ref{ssharmcoc} and \S{}\ref{ssbanachrep}) and then project them (\S{}\ref{ssgrad}). The second is by a trick known as virtual coboundary (\S{}\ref{ssvirtcob}). 
On the way we show that in the mildly mixing setup virtual coboundaries are actually quite generic. \S{}\ref{ssvirtcobtree} gives an example of virtual coboundaries outside the setup of \S{}\ref{ssvirtcob}.

\subsection{Unitary representation and harmonicity}\label{ssharmcoc}

The proofs of this subsection may be found in Guichardet \cite[\S{}3-\S{}5]{Gu} and \cite[\S{}2]{GJ}.
There is a natural topology and, sometimes, a natural Banach space structure which can be put on $Z^1(G,\pi)$ which enables us to study it better.
For the rest of this subsection, it will be assumed that $G$ is a countable group which is generated by some finite subset $S$ with $S^{-1} = S$. 

In the next subsection, more general representations will be discussed while this subsection is devoted to the case of unitary representations (Hilbert spaces).
Thanks to the cocycle relation, a cocycle $b$ is completely determined by the values $\{b(h)\}_{h\in S}$. 
Given an measure $\mu$ with support $S$, it would be natural to introduce the following scalar product:
\[
(b \mid b')_{\mu}= \sum_{s\in S} \mu(s) \langle b(s),b'(s)\rangle,
\] 
for $b,b'\in Z^1(G,\pi)$. Since $b$ is identically $0$ on $G$ if and only if $b$ is identically $0$ on $S$, this scalar product is be non-degenerate. 

To be sure that this scalar product defines a Hilbert space, we need to check that the space $Z^1(G,\pi)$ is complete for the norm $\|\cdot\|_\mu$. 
The classical topology on $Z^1(G,\pi)$ is given by the topology of uniform convergence on compact subsets of $G$. 
Since $G$ is countable and endowed with the discrete topology, it is well-known that this topology turns $Z^1(G,\pi)$ into a Fr\'echet space. 
See Guichardet \cite[\S{}4]{Gu} or \cite[\S{}2]{GJ} for proofs and discussions that these topologies coincide.

When $G$ is not finitely generated, it is not completely impossible to go further. 
It is possible, given any cocycle $b$ to define a $\mu$ so that $\|b\|_\mu < \infty$ and $b$ is in the $\|\cdot\|_\mu$-norm closure of $Z^1_\mu(G,\pi)$. 
But there is in general no way of picking a $\mu$ which works for all cocycles.

So let $\mu$ be some fixed measure on the finite generating set $S$ and let $d$ be the coboundary map. Simple computations give:
\[
d_\mu^* b = -\sum_{s\in S} \mu(s) (b(s^{-1})+b(s)).
\] 
By generic considerations on Hilbert spaces, $Z^1(G,\pi)= \ker d_\mu^* \oplus (\ker d_\mu^*)^\perp$ and 
the orthogonal complement of $\ker d_\mu^*$ coincides with $\srl{\img d}$. 
Moreover, this latter space is just the closure of $B^1(G,\pi)$ with respect to the norm $\|\cdot\|_\mu$ (or, equivalently, with respect to the topology of uniform convergence on compact sets). 
It will henceforth be denoted by $\srl{B}^1(G,\pi)$.
Therefore, one has the following general orthogonal decomposition.
\begin{prop}
Let $G$ be finitely generated by $S$, then
\[
Z^1(G,\pi)=\ker d_\mu^* \oplus \srl{B}^1(G,\pi).
\]
\end{prop}
The space $\ker d_\mu^*$ is the space of $\mu$-\textbf{harmonic cocycles}. 

A first straightforward consequence is the description of the reduced cohomology in terms of harmonic cocycles. 
Recall that the \textbf{reduced cohomology} group of $G$ taking value in $\pi$ is defined as the quotient space
\[
\srl{H}^1(G,\pi)=Z^1(G,\pi)/\srl{B}^1(G,\pi).
\]
\begin{cor}\label{tidharcocredcoh-c}
There is an isomorphism $\srl{H}^1(G,\pi)=\ker d_\mu^*$. 
\end{cor}
$\pi$ is said to have \textbf{no almost invariant vectors} when $B^1(G,\pi) = \srl{B}^1(G,\pi)$ (equivalently, when $d$ has closed image). Note that when this is the case, then $H^{1}(G,\pi)=\ker d_\mu^*$

\subsection{Gradient conditions}\label{ssgrad}

It turns out that, when considering the left-regular representation, both the left- and right-Cayley graphs come up naturally. 
Assume $S$ is a symmetric generating set for $G$. 
Recall that the left-\textbf{Cayley graph} (resp. right-) of $G$ with respect to $S$, denoted $\cayl(G,S)$ (resp. $\cayr(G,S)$), is the graph whose vertex set is $G$ and whose edge set is $E = \{ (x,y) \in G \times G \mid \exists s \in S$ so that $s x= y\}$ 
(resp. 
so that $xs= y\}$). 

Let $V$ be Banach space (usually $\rr$ or $\cc$). 
Given a function $f:G \to V$ and a Cayley graph, the \textbf{gradient} of this function is the function $\nabla f: E \to V$ defined by $\nabla f (x,y) = f(y) - f(x)$.

It turns out $\nabla: \ell^p (G,V) \to \ell^p (E,V)$ is bounded exactly when $S$ is finite. 
There is a natural pairing for finitely supported functions $f: X \to V$ and $g:X \to V^*$ (where $X = G$ or $E$) given by $\pgen{f \mid g} = \sum_{x \in X} f(x) g(x)$. 
In the case $S$ is finite, $\nabla$ then has a adjoint $\nabla^*$ which associates to $g:E \to V^*$ the function $\nabla^*g: G \to V$ defined by $\nabla^* g (x) = \sum_{y \sim x} g(y,x) - g(x,y)$.

Given a subspace $\XX$ of functions $G \to V$ (\eg $\ell^p$, WAP, etc..) Say that a function has $\XX$-gradient (on $\cayr(G,S)$) if for any $s \in S$, the function $x \mapsto g(xs) - g(x)$ belongs to $\XX$.

Note that when $S$ is infinite, $\nabla f \in \ell^p E$ implies $f$ has $\ell^p$-gradient; but the converse is false.

A representation $\pi$ has \textbf{finitarily coefficients}\footnote{For our actual purposes, it is sufficient to assume this works for $n$ equal to the minimal number of elements required to generate $G$ as a group.}
in $\XX$ if for any $n \in \nn$ and $\xi_1,\ldots,\xi_n \in \Hi$ there is a $\eta  \in \Hi$, so that $\kappa_{\eta,\xi_i}$ belongs to $\XX$ and for some $i$, $\kappa_{\eta,\xi_i} \not\equiv 0$. 
This holds if there is a dense subspace $\Hi' \subset \Hi$ so that for any $\xi \in \Hi$ and $\eta  \in \Hi' \setminus \{0\}$, $\kappa_{\eta,\xi}$ belongs to $\XX$.
For example, the left-regular representation has finitarily coefficients in $\ell^2$ (let $\Hi' \subset \ell^2G$ be the space of finitely supported functions).

Weakly mixing implies finitarily coefficients in $\XX = \ker m$ (where $m$ is the unique mean on WAP functions) and strongly mixing implies it for $\XX = c_0$. Having finitarily coefficients in $\XX = \ell^p$, is most probably stronger than strongly mixing.

A function $f$ on a [right-]Cayley graph is called $\mu$-harmonic if $\sum_{s \in S} \mu(s) \nabla f(x,xs) =0$.
\begin{lem}\label{tgradcond-l}
Assume $G$ is finitely generated, $\mu$ is symmetric (\ie $\mu(s) = \mu(s^{-1})$) and $b \in Z^1(G,\pi)$ is a harmonic cocycle. Then, for any $\eta \in \Hi$, the function $h_\eta(g) = \pgen{b(g) \mid \eta }$ is $\mu$-harmonic (on $\cayr(G,S)$).

If $\pi$ has finitarily coefficients in $\XX$, then there is a choice of $\eta$ so that $h_\eta: G \to \cc$ has $\XX$-gradient and is not constant.
\end{lem}
\begin{proof}
Note that 
\[
h_\eta(xs)-h_\eta(x) = \pgen{ b(xs) - b(x) \mid \eta } = \pgen{ \pi(x)b(s) \mid \eta }.
\]
When $\mu$ is symmetric $d_\mu^* b =0 \iff \sum_{s \in S} \mu(s)b(s) =0$. Since
\[
\sum_{s} \mu(s) \big( h_\eta(xs)-h_\eta(x) \big) = \sum_{s} \mu(s) \pgen{ b(xs) - b(x) \mid \eta } = \pgen{ \pi(x) \sum_{s} \mu(s)b(s) \mid \eta } =0,
\]
$h_\eta$ is $\mu$-harmonic.

Next, let $\xi_s := b(s)$ for $s \in S$. Since $b$ is not trivial, at least one of the $\xi_s$ is not trivial. 
Since $\pi$ has finitarily coefficients in $\XX$, there is a choice of $\eta$ (with respect to $\{\xi_s\}_{s \in S}$) so that $h_\eta$ has gradient in $\XX$ and is not trivial.
\end{proof}

\subsection{Banach representations}\label{ssbanachrep}

The aim of this section is to show that, if one is ready to consider some non-linear brethren of harmonicity, there is also an isomorphism for the reduced cohomology of representations in strictly convex Banach space. 

Let $\pi: G \to \LI(\BB)$ be a linear isometric representation of $G$ on the strictly convex Banach space $\BB$. 
Throughout this subsection, $G$ is assumed finitely generated and $\mu$ is as before (symmetric and supported on a finite generating set). 

Introduce, for some $p \in ]1,\infty[$
\[
\| b\|_\mu = \bigg( \sum_s \mu(s) \| b(s) \|_\BB^p \bigg)^{1/p}
\]
The same argument as in the Hilbertian case shows this norms induces the same topology as the the topology of uniform convergence on compacts.

Recall from Benyamini \& Lindenstrauss \cite[Proposition 4.8 and Appendix A]{BL} that if $\BB$ is strictly convex then
\[
\frac{ \mr{d}}{\mr{d}t} \| x+ t y\|_\BB \bigg|_{t=0}  = \pgen{j(x) \mid y}
\]
where $j(x)$ is the unique element of $\BB^*$ with $\|j(x)\|_{\BB^*} = 1$ and $\pgen{j(x) \mid x } = \|x\|_\BB$. 
Sometimes, $j$ is called the duality map. Given the norm introduced above, it will be more convenient to speak of 
\[
j_p(x) := \|x\|^{p-1}_\BB j(x).
\]
From there, one sees the exponent $p$ essentially only change the homogeneity of the function $j$. 
It is natural to pick $p \neq 2$ when $\BB = \L^p_0(X,\nu)$ and $\pi$ is an $\L^p$-representation of $G$.
\begin{lem}
Given $b \in \srl{H}^1(G,\pi)$, then $b$ has minimal $\|\cdot\|_\mu$-norm in its reduced cohomology class if and only if $b$ is $(j_p,\mu)$-harmonic,  \ie
\[
\sum_s \mu(s) j_p(b(s)) =0
\]
\end{lem}
\begin{proof}
Note that the norm is strictly convex, hence for any $b' \in Z^1(G,\pi)$ there is a unique $b \in b' + \srl{B}^1(G,\pi)$ with minimal norm. 
Furthermore, by minimality, for any $z \in B^1(G,\pi)$.
\[
\begin{array}{rl}
\displaystyle 
0=  \frac{ \mr{d}}{\mr{d}t} \| b+ t z\|_\mu \bigg|_{t=0} 
& \displaystyle = \frac{ \mr{d}}{\mr{d}t} \bigg( \sum_s \mu(s) \| b(s) + tz(s) \|_\BB^p \bigg)^{1/p} \bigg|_{t=0}  \\
& \displaystyle = \frac{1}{p} \bigg( \sum_s \mu(s) \| b(s) + tz(s) \|_\BB^p \bigg)^{\tfrac{1}{p}-1} \bigg( \sum_s \mu(s) \frac{ \mr{d}}{\mr{d}t} \| b(s) + tz(s) \|_\BB^p \bigg) \bigg|_{t=0}  \\
& \displaystyle = \| b+ tz\|_\mu^{- (p-1)/p } \bigg( \sum_s \mu(s) \| b(s) + tz(s) \|_\BB^{p-1} \frac{ \mr{d}}{\mr{d}t}  \| b(s) + tz(s)  \|_\BB \bigg)\bigg|_{t=0} \\
& \displaystyle = \| b\|_\mu^{- (p-1)/p } \sum_s \mu(s) \pgen{ j_p(b(s)) \mid z(s)}.
\end{array}
\]
Writing $z(s) = \xi - \pi(s) \xi$ where $\xi \in \BB$, one gets
\[
0 = \sum_s \mu(s) \pgen{ j_p(b(s)) \mid \xi - \pi(s)\xi} = \sum_s \mu(s) \pgen{ j_p(b(s)) - \pi(s)^* j_p(b(s)) \mid \xi},
\]
where $\pi(\cdot)^*$ is the adjoint representation on $\BB^*$. 
If one knows that $- \pi(s)^* j_p(b(s)) = j_p( b(s^{-1}))$, then the proof is over: for any $\xi \in \BB$, one has
\[
2 \sum_s \mu(s) \pgen{ j_p(b(s)) \mid \xi} =0.
\]
To prove this claim, note that $j_p( b(s^{-1}) ) = j_p ( - \pi(s^{-1}) b(s) ) = - j_p (\pi(s^{-1}) b(s) )$. So it remains to prove that $j_p (\pi(s^{-1}) b(s) ) = \pi(s)^* j_p(b(s))$. To do so it suffices to check that the two defining properties of $j_p$ are satisfied: 
\begin{enumerate} \renewcommand{\itemsep}{-1ex}
 \item $\pgen{ j_p(\eta) \mid \eta } = \|\eta\|^p_\BB$ 
 \item and $\| j_p(\eta) \|_{\BB^*} = \|\eta\|^{p-1}_\BB$.
\end{enumerate}
For 1:
$
\pgen{ \pi(s)^* j_p(b(s)) \mid \pi(s^{-1}) b(s) } = \pgen{  j_p(b(s)) \mid b(s) } = \| b(s) \|_\BB^p = \| \pi(s^{-1}) b(s) \|_\BB^p.
$ \\
As for 2:
$
\|\pi(s)^* j_p(b(s)) \|_{\BB^*} 
= \|j_p(b(s)) \|_{\BB^*}
= \| b(s) \|_{\BB}^{p-1}
= \| \pi(s^{-1}) b(s) \|_{\BB}^{p-1}.
$
\end{proof}
For a $\L^p$-representation and $\mu$ a uniform measure on $S$, the $(j_p,\mu)$-harmonic cocycles are called $p$-harmonic cocycles.

\subsection{Virtual coboundaries}\label{ssvirtcob}

Although the virtual coboundaries will only pop-up in the present work for the left-regular representation (on $\ell^pG$), it is a fairly common method to make a cocycle with some desirable properties, see \eg Fernos \& Valette \cite{FV} or \cite{Go-comp}. As such it is pertinent to ponder on how general this method is. In this subsection, it will be shown that this a general phenomena at least for mildly mixing for $L^p$-representations (or more generally, those with finite stabilisers).

In order to speak of virtual coboundary, the Banach (or Hilbert) space $\BB$ has to lie in a vector space $W$ and the unitary representation $\pi$ has to extend to a linear group action on $W$. 
A cocycle is then a virtual coboundary if $b(g) = \pi(g) x - x$ for some $x \in W \setminus \BB$.

In this section, only $\L^p$-representations will be considered. 
(One could also consider $c_0(X,\mu)$ the closure under $L^\infty$ of the compactly supported functions.)
The natural choice for the space $W$ is simply the whole set of functions on $X$, \ie $W = \{ f: X \to \mathbb{C}\}$.

In order to give simple conditions which allow to realise $b$ as a virtual coboundary, it is useful to consider to think of the Schreier graph of the action for some generating set. 
The vertices of the [left-]Schreier graphs are the elements of $X$ and $y \sim x$ if there is a $s \in S$ so that $y = s^{-1}x$. 
For the left-regular representation, the [left-]Schreier graph is the [left-]Cayley graph.

Denote $b(g;x)$ for the function $b(g) \in \L^p_0(X,\mu)$ evaluated at $x$.
\begin{lem} \label{tvirtcbndr-l}
Assume $b$ is a cocycle for $\pi$ the permutation representation on $\L^p_0(X,\mu)$. There is a $f:X \to \mathbb{C}$ such that $b(g) = \pi(g) f - f$ if and only if, there is a $x_0$ in each orbit (equiv. for all $x_0$) so that $b(k;x_0) =0$ for all $k \in \Stab(x_0)$. 
\end{lem}
\begin{proof}
The $(\Rightarrow)$ direction is straightforward so we will only deal with the $(\Leftarrow)$.
The cocycle relation implies that $b(g;x)$ is completely determined by $b(s;x)$ for $s \in S$ (by writing $g$ as a word). 

Hence a first step is to find a $f$ such that $b(s;x) = \pi(s) f(x) - f(x)$ for any $s \in S$. This turns out to be equivalent to fixing the gradient of $f$ on the [left-]Schreier graph. The fact that $b(e) \equiv 0$ and the cocycle relation imply that $\sum_{\vec{e} \in C} \nabla f(\vec{e}) =0$ (where $C$ is a [oriented] cycle) so that the gradient can be ``integrated'' into a $f:X \to \mathbb{C}$. $f$ is uniquely determined on the [left-]Schreier graph up to a constant on each $G$-orbit. 

The next step is to make sure that there are no problems coming from the stabilisers, namely that $b(w;x) =0$ if $w \in \Stab(x)$ (because $\pi(w) f(x) - f(x) =0$). 
Write $w = g k g^{-1}$ where $x = gx_0$ and $k \in \Stab(x_0)$. Then
\[
\begin{array}{rll}
b(w) &= b(gkg^{-1}) 
     &= b(g) + \pi(g) b(kg^{-1}) \\
     & = b(g) + \pi(g) b(k) + \pi(gk) b(g^{-1}) 
     &= [1-\pi(gkg^{-1})] b(g) + \pi(g)b(k)
\end{array}
\]
where we used $b(g^{-1}) = -\pi(g^{-1})b(g)$ in the last inequality. Notice that $[1-\pi(gkg^{-1})] h (x) =0$ for any function $h:X \to \mathbb{R}$ (because $gkg^{-1}=w$ fixes $x$). Also $\pi(g)b(k;x) = b(k;g^{-1}x) = b(k;x_0) = 0$. Hence, we have $b(w;x) =0$ as desired.
\end{proof}
\begin{rmk}\label{rvirtcobharm}
Note that if $b$ is harmonic, then $f$ is harmonic as a function on the [left-]Schreier graph, \emph{i.e.} $\forall x \in X, \sum_{s \in S} f(s^{-1}x) -f(x) = 0$.

Indeed, the condition $\forall x, \sum_{s \in S} b(s;x) =0$ is equivalent to $\sum_{s \in S} (f(s^{-1}x)-f(x) ) =0$ (so the two notions of harmonicity coincide).

The same holds for ``$p$-harmonic'' instead of harmonic.
\end{rmk}
Then $K$ is a compact subgroup of $G$, it is always possible to add a coboundary $z_\xi(\cdot) := \xi - \pi(\cdot) \xi$ to a cocycle $b$ in order to have $b_{|K} \equiv 0$. The coboundary is given by $\xi = \int_K b(k)$. If $b$ vanishes identically on $K$, one even has that $f$ is constant on the $K$-orbits: $\forall x$ one has $0= b(k;x) = f(k^{-1}x) - f(x)$.

The following lemma is however better suited to our purposes.
\begin{lem}\label{tstabcoc-l}
Assume $G \acts X$ and $K = \Stab(x_0)$. If $b$ is a cocycle for the permutation representation $\pi$, then $k \mapsto b(k;x_0)$ is a homomorphism $K \to \L^p(X,\mu)$. 
\end{lem} 
In particular, if $K$  is compact, then $b(k;x_0)=0$ for all $k \in K$.
\begin{proof}
By the cocycle relation
\[
\begin{array}{ll}
b(k^{-1}k';x_0) 
&= b(k^{-1};x_0) + \pi(k^{-1}) b(k';x_0) \\
&= b(k^{-1};x_0) + b(k';k x_0) \\
&= b(k^{-1};x_0) + b(k';x_0). \\
\end{array}
\]
\end{proof}
In the following theorem, the author benefited from the help of A.~Carderi to deal with the case of a generic measure space $(X,\mu)$. 
\begin{teo}\label{tvirtcoc-t}
Assume $G \acts X$ countable and the action is mildly mixing, then any cocycle is a virtual coboundary.

If $G \acts (X,\mu)$ and the action is mildly mixing (as in \ref{defmixact}), then, on a full measure subset $X' \subset X$, $b$ is a virtual coboundary.
\end{teo}
Actually, it is sufficient to assume that the representation has finite stabilisers (\ie if $A \subset X$ is such that $\mu(A) >0$ and $\mu(A^\comp) >0$ then there are only finitely many $g \in G$ with $g A = A$).
\begin{proof}
Let us start with the case where $X$ is countable. Assume $b$ is not a virtual coboundary. By Lemma \ref{tvirtcbndr-l}, there is a stabiliser $K = \Stab(x_0)$ so that $b(k;x_0) \neq 0$. By Lemma  \ref{tstabcoc-l}, $K$ must be infinite so the action is not mildly mixing.

In the case of the action on $(X,\mu)$, let $E = \{g \in G \mid b(g;x_0) \neq 0$ and $g \in \Stab(x_0)\}$. 
If this set is empty, Lemma \ref{tvirtcbndr-l} can be used to conclude directly. 
By Lemma \ref{tstabcoc-l}, any element of $E$ has an infinite order. 
Let $F_g = \cup \{ A \subsetneq X \mid g \cdot A = A \text{ and } \mu(A^\comp) >0 \}$ be the ``largest'' fixed set (it contains at least one point). 
Then, for $g \in E$,  $\mu(F_g) =0$ (otherwise there is a $A$ with $\mu(A)>0$ and $\mu(A^\comp)>0$ so that $g^n A = A$ for all $n$, contradicting mildly mixing). 
Let $\tilde{X} = \cap_{g \in E} (X \setminus F_g)$. 
Since $E \subset G$ is countable, $\tilde{X}$ has full measure. 
It could very well happen that $\tilde{X}$ is not $G$-invariant, so one needs to consider $X' = \cap_{g \in G} g \cdot \tilde{X}$.

By construction, $X'$ has full measure and $b_{|X'}$ is a virtual coboundary by Lemma \ref{tvirtcbndr-l}.
\end{proof}

It seems natural to introduce the space of $p$-Dirichlet functions associated to the representation. These function have been of great use to study the cohomology of the left-regular representation, see \cite{Go}, Martin \& Valette \cite[\S{}3]{MV}, Puls \cite{Puls-Can,Puls-harm} or \S{}\ref{slpcoh} below.
For $S_G$ some finite generating set of $G$, let 
\[
\D^p(G,\pi) = \{ f : X \to \cc \mid \forall s \in S_G , \pi(s) f - f \in \L^p(X,\mu)\}.
\]
By the previous subsection, the space $\D^p$ (modulo constants on the $G$-orbits) is the space of cocycles. 
Though the above definition depends on $S$, the important properties (\ie those related to cohomology) obviously do not.
Note that an element $f$ of $\D^p(G,\pi)$ might not be a measurable function, but $\pi(g)f - f$ always is.

\begin{rmk}
It is easy to describe the action introduced in Remark \ref{ractnorm} in terms of the virtual coboundary. If $K \lhd G$, then $G$ acts on elements $f \in \D^p(K,\pi)$ by $g \cdot f = \pi(g)f$.
\end{rmk}

\begin{rmk}
Given a cocycle for a $L^2$-representation $\pi$, there are actually two families of harmonic functions one can associate to it. 
The first, coming from projections (see Lemma \ref{tgradcond-l}), are ``naturally'' on the right-Cayley graph and the second, coming from virtual coboundaries (see Remark \ref{rvirtcobharm}), are ``naturally'' on the left-Schreier graph. 

For the left-regular representation, the Schreier graph is just the Cayley graph.
The relation between those left- and right-harmonic functions are easily seen for any $\eta \in \ell^2G$ of finite support:
\[
h_\eta(g) = \pgen{ b(g) \mid \eta } = \pgen{ \pi(g) f - f \mid \eta }  = \pgen{ f \mid \pi(g)^{-1} \eta - \eta}.
\]
For example, when $\eta = \delta_x$ is a Dirac mass, $h_{\delta_x}(g) = f(g^{-1}x)-f(x)$. 
Note that the constant in the definition of $f$ does not matter, so one may set $f(x)=0$. 
Also, one sees that the choice of $x$ only translates the function $h$ (by an automorphism of the right-Cayley graph), so one could set $x=e$.
Then $h$ and $f$ are actually obtained by the change of variable $g \mapsto g^{-1}$, \ie the most obvious way of passing from right- to left-Cayley graphs.
\end{rmk}

\subsection{Virtual coboundary for groups acting on trees}\label{ssvirtcobtree}

There are clearly cases outside those described in the previous subsection where virtual coboundaries are useful. 
In this subsection, the setting is that of \cite{GJ}; as such the groups may not be finitely generated. 
These are groups acting on trees by automorphism.
However, there is a compact subgroup (the stabiliser of some ``root'' vertex $x_0$) that has only countably many co-sets.
This compact group together with a finite number of elements generate the group.

The action of interest for such group is the action on the space $\ell^2_{alt}(E)$ of $\ell^2$-alternating functions on oriented edges (\ie $f(\vec{e}) = -f(\vec{e}^*)$ where $\vec{e}^*$ is the edge $\vec{e}$ with reversed orientation). 
The space $\ell^2_{alt}(E)$ is obviously contained in the space of all alternated functions, playing the r\^ole $W$ above.
Note that in this context, there are no constant functions in $W$!

Let us now describe a virtual coboundary for the Haagerup cocycle (see \cite{GJ} for further background and see \cite[\S{}3]{FV} for another possible choice). 
A simple computations shows that $f$ may be defined as follows:
\[
f(\vec{e}) = \left\{\begin{array}{ll}
+1/2 & \text{if following } \vec{e} \text{ increases the distance to } x_0\\
-1/2 & \text{if following } \vec{e} \text{ decreases the distance to } x_0\\
\end{array} \right.
\]
In order to make it harmonic, one looks first at its divergence: it is $\tfrac{q+1}{2}$ at $x_0$ and $\tfrac{q-1}{2}$ at every other vertex. 
Recall that each [non-trivial] instance $b(g) = \pi(g) f- f$ of the cocycle is a function with non-trivial divergence only at two vertices $x_0$ and $gx_0$.

The harmonic cocycle (in the same class as the Haagerup cocycle) is obtained from the Haagerup cocycle by adding to it a bounded cocycle (hence a coboundary): $\tilde{b}(g) = Q \chi_{x_0 \to g x_0}$ where $Q= \nabla G \nabla^*$ and $G$ is Green's kernel. The cocycle $\tilde{b}$ is the image under the boundary map of $\tilde{f}: \mathbb{E} \to \mathbb{C}$, the function defined as the gradient of $G$. 

It is then apparent that the harmonic cocycle $b'$ is also a virtual coboundary: $b'(g)= \pi(g) f' - f'$ with $f' = f+\tilde{f}$. 
Note that the function $f'$ has divergence $\tfrac{q-1}{2}$ at every vertex, so that $b'(g) = \pi(g) f'- f'$ lies indeed in the kernel of the divergence.

On a tree all these [alternated] function on the edges can be integrated as functions on the vertices. 
For example, $f$ is the gradient of the function $x \mapsto \tfrac{1}{2} d(x,x_0)$. 
Note that $f'$ being of constant divergence and spherical, its integral will have constant Laplacian and be spherical. 
In particular, this means it will satisfy the recurrence relation described in \cite[\S{}3.1]{GJ} and will be easily computed.

\section{Centres and vanishing}\label{scen}

In this section, only unitary representations are considered. The aim is to try to deduce vanishing of reduced cohomology by exploiting some specific properties of harmonic cocycles. This property is sometimes extremely useful for this purpose (see Ozawa's recent proof of Gromov's polynomial growth theorem  \cite{Oza}). The proofs of Lemma \ref{tergconj-l} and Lemma \ref{twmconj-l} are inspired from \cite[Theorem 3.2]{Go-trans} and \cite[Lemma 2.7]{GJ}.

Recall that for some set $S \subset G$,
\[
\begin{array}{rlll}
           & Z_G(S)      &= \{ g \in G \mid \forall h \in S, h^{-1} g h = g \}          &=\text{ the centraliser of } S \\
\text{and} & Z^{FC}_G(S) &= \{ g \in G \mid \{h^{-1}gh \}_{h \in S} \text{ is finite}\} &=\text{ the $FC$-centraliser of } S.
 \end{array}
\]
The centre of $G$ is $Z(G):= Z_G(G)$ and the FC-centre of $G$ is $Z^{FC}(G) := Z^{FC}_G(G)$. 
Much like the FC-centre is not very meaningful in finite groups, the FC-centraliser of a finite set is also the whole group.
Note that having an infinite FC-centre is not an invariant of quasi-isometry (see Cornulier \cite[Remark 2.14]{Cor-comm}).

\begin{rmk}
For an infinite subgroup $K< \Gm$,
note that the FC-centraliser $Z^{FC}_\Gm(K)$ is contained in the q-normaliser $N^q_\Gm(K)$. 
Indeed, given $g \in Z^{FC}_\Gm(K)$, $K$ acts on $\{ k g k^{-1}\}_{k \in K}$ by conjugation (and there is, by construction, only one orbit). 
Since this set is finite and $K$ is infinite, every orbit has an infinite stabiliser. 
So there are infinitely many $k \in K$ commuting with $g$ and $gKg^{-1} \cap K$ is infinite.

Hence, if $K<H<\Gm$ (and $K$ is an infinite subgroup), then $Z^{FC}_\Gm(H) < Z^{FC}_\Gm(K) < N^q_\Gm(K) < N^q_\Gm(H)$.
\end{rmk}

\subsection{FC-centre and kernel}\label{sscentraliser}

$G$ is called $\mu$-Liouville if there are no bounded $\mu$-harmonic functions on $G$, \ie there are no $f \in \ell^\infty G$ so that $\sum_{s \in S} \mu(s) f(s^{-1}x) = f(x)$. 
\begin{lem}\label{tergconj-l}
Assume $\pi$ is an ergodic unitary representation of the finitely generated group $G$, $b$ is a $\mu$-harmonic cocycle, $G$ is $\mu$-Liouville and $G$ has a finite conjugacy class $C$. Then 
\[
\frac{1}{|C|}\sum_{c \in C} b(cg) = b(g). \qedhere
\]
\end{lem}
\begin{proof}
Define a transport pattern (see also \cite[Definition 3.2]{Go} or \cite [Definition 4.1]{Go-trans}) from $\phi$ to $\xi$ (two finitely supported measures) to be a finitely supported function on the edges $\tau$ so that $\nabla^* \tau = \xi - \phi$. 
The proof consists in showing the equality for any $h(\cdot)$ which is $\mu$-harmonic and has ``$0$-mean''-gradient. It will then follow for $b$, since for any $\eta$, $h(\cdot) := h_\eta(\cdot) = \pgen{b(\cdot ) \mid \eta}$ such a function (see Definition \ref{defmix} and Lemma \ref{tgradcond-l}). The proof takes place on the right-Cayley graph.

If $h$ is $\mu$-harmonic, then $\langle h \mid P^n_g \rangle = h(g)$ where $\langle h \mid f \rangle = \sum_{x \in X} h(x)f(x)$ (if at least one of $h$ or $f$ has finite support) and $P^n_g$ is the distribution of the random walk driven by $\mu$ starting at $g$ at time $n$. Hence, if $C$ is a finite conjugacy class,
\[
\frac{1}{|C|} \sum_{c \in C} h(cg) - h(g) = \Big\langle h \Big| \tfrac{1}{|C|} \sum_{c \in C} P^n_{cg} - P^n_g \Big\rangle = \langle h \mid \nabla^* \tau_n \rangle = \langle \nabla h \mid \tau_n \rangle
\]
where $\tau_n$ is the transport plan obtained by taking the mass of $P^n_g$ at $g'$, split in $|C|$ masses, and take them (along a shortest path) to $g'c$ (for $c \in C$). Each transport takes at most $K:= \max_{c \in C} |c|_S$ steps, hence $\|\tau_n\|_{\ell^1} \leq K$. Notice that $g'C = Cg'$, so that this also splits and transports the mass uniformly to $Cg'$. Let $c = s_{1,c} s_{2,c} \ldots s_{|c|,c}$, then
\[
\begin{array}{r@{\,}l}
\displaystyle \tfrac{1}{|C|} \sum_{c \in C} h(cg) - h(g)  
&= \langle \nabla h \mid \tau_n \rangle \\
&= \displaystyle  \sum_{c \in C} \sum_{i = 1}^{|c|} \sum_{g' \in G} \pgen{ \nabla h( g' s_1 s_2 \ldots s_{i-1} , g' s_1 s_2 \ldots s_i ) \mid \tfrac{1}{|C|} P^n_{g}(g') } \\
&= \displaystyle \tfrac{1}{|C|}\sum_{c \in C} \sum_{i = 1}^{|c|} \sum_{g' \in G}  \pgen{\pi( g' s_1 s_2 \ldots s_{i-1}) b(s_i )\mid \eta} \cdot  P^n_{g}(g')  \\
&= \displaystyle \tfrac{1}{|C|} \sum_{c \in C} \sum_{i = 1}^{|c|} \pgen{ \kappa_{\pi(s_1 \ldots s_{i-1}) b(s_i),\eta } \mid P^n_{g} } \\
\end{array}
\]
Note that the left-hand side does not depend on $n$. By letting $n \to \infty$, measure $P^n_g$ in the right-hand side tends to an invariant mean on $\ell^\infty G$ (see Kaimanovich \& Vershik \cite{KV}). 
On the other hand, the coefficients are those of an ergodic representation: their mean for the invariant mean on WAP functions is $0$. Since an invariant mean on $\ell^\infty G$ coincides with the unique invariant mean on WAP, each sum tend to $0$.
\end{proof}
Recall that the [closure of the] image of a cocycle is always a subrepresentation.  
\begin{lem}\label{tergtrivcent-l}
Assume that $\exists z$ so that $\forall g \in G, b(g) = b(zg)$. Then $z \in \ker \pi_{| \img b}$ and $b(z) =0$.
\end{lem}
\begin{proof}
Using the hypothesis and that $b(e) =0$, one gets
\[
b(z) = b(z \cdot e) = b(e) =0
\]
But then, using the cocycle relation, one has, $\forall g$
\[
b(g) = b(zg) = \pi(z) b(g). \qedhere
\]
\end{proof}
For the record let us combine the two previous lemmas:
\begin{teo}\label{tergliou-t}
Assume $\pi$ is ergodic, $\srl{H}^1(G,\pi) \neq 0$, $G$ is finitely generated and $G$ is $\mu$-Liouville. Then, for any non-trivial $\mu$-harmonic cocycle $b$, $Z(G) \subset \ker b \cap \ker \pi_{| \img b}$.
\end{teo}
\begin{proof}
Let $z \in Z(G)$.
Then $z$ has a finite conjugacy class (it consists in the singleton $z$ itself!).
By Lemma \ref{tergconj-l}, $b(zg) = b(g)$ for all $g \in G$.
By Lemma \ref{tergtrivcent-l}, $z \in \ker \pi_{| \img b}$ and $b(z) =0$.
This finishes the proof.
\end{proof}

Here is a small strengthening of Lemma \ref{tergtrivcent-l}. Note that the conclusion on the triviality of the cocycle goes through only for weakly mixing representations.
\begin{lem}\label{twmtrivfc-l}
Assume $C \subset G$ is a finite subset and that $\forall g \in G, b(g) = \tfrac{1}{|C|} \sum_{c \in C} b(cg)$ then $C \subset \ker \pi_{| \img b}$.
If further $C$ is a finite conjugacy class, then $\{b(c) \mid c \in C\}$ span a finite-dimensional subrepresentation of $\pi$ (possibly $b(C) =\{0\}$).
\end{lem}
\begin{proof}
Note that 
\[
0 = b(e) = \tfrac{1}{|C|} \sum_{c \in C} b(c).
\]
Using this, one gets
\[
b(g) = \tfrac{1}{|C|} \sum_{c \in C} b(cg) = \tfrac{1}{|C|} \sum_{c \in C} \pi(c)b(g).
\]
So $\tfrac{1}{|C|} \sum_{c \in C} \pi(c) = \Id_{| \img b}$. For any $\xi \in \img b$,
\[
\|\xi\|^2_\Hi = \tfrac{1}{|C|} \sum_{c \in C} \pgen{  \pi(c) \xi \mid \xi }.
\]
Since $ \pgen{  \pi(c) \xi \mid \xi } \in [-\|\xi\|^2_\Hi,\|\xi\|^2_\Hi]$, each term of the average on the right-hand side must be equal to $\|\xi\|^2$. Using the classical trick that $\|\xi -\pi(c)\xi\|^2 = 2 \|\xi\|^2 - 2 \pgen{\pi(c) \xi \mid \xi}$, one gets that $\pi(c) \xi = \xi$. Hence $C \subset \ker \pi_{| \img b}$.

Using \eqref{eqcocconj} with $h =c$ and noting that $\pi(gcg^{-1})$ is trivial on $\img b$, one gets that 
\[
b(gcg^{-1}) = \pi(g)b(c).
\]
This implies the second claim.
\end{proof}
The previous lemma will be shortly applied to weakly mixing representations. 
In order to obtain a stronger result than Theorem \ref{tergliou-t}, the Liouville hypothesis will be dropped.
Since the random walk no longer converges to an invariant mean, some preliminary work on the coefficients of weakly mixing representations needs to be done.
\begin{lem}\label{twmsynd-l}
Assume $\pi$ is a weakly mixing representation of $G$.
Let $\kappa_1, \ldots, \kappa_n$ be coefficient functions of $\pi$.
For any $\eps >0$ the set $T_\eps = \{ g \in G \mid \text{for all } i, |\kappa_i(g)| <\eps \}$ is thickly syndetic and will be hit almost surely infinitely many times by a [irreducible] random walk. 
\end{lem}
\begin{proof}
For any $\eps >0$ the sets $T_i = \{ g \in G \mid |\kappa_i(g)| <\eps \}$ are thickly syndetic (see Definitions \ref{defsynd} and \ref{defmix}).
The first step is to prove that the intersection of finitely many thickly syndetic set is thickly syndetic.

This can be obtained by induction using the following claim: if $A$ is thickly syndetic and $B$ is thickly syndetic, then $A \cap B$ is thickly syndetic. 
To prove the claim, consider first the simpler case where $A$ is syndetic. Let $S$ be such that $SA = G$, then 
\[
\begin{array}{rll}
\displaystyle \bigcap_{t \in S} tB 
&= \displaystyle G \cap \big( \bigcap_{t \in S} tB \big) 
&= \displaystyle \big( \cup_{s \in S}  sA \big) \cap \big( \bigcap_{t \in S} tB \big)\\
&= \displaystyle \bigcup_{s \in S} \bigg( sA  \cap \big( \bigcap_{t \in S} tB \big) \bigg)
&= \displaystyle \bigcup_{s \in S} s \bigg( A  \cap \big( \bigcap_{t \in S} s^{-1}tB \big) \bigg)\\
& \subseteq \displaystyle  \bigcup_{s \in S} s \big( A  \cap B \big) 
\end{array}
\]
Since $\bigcap_{t \in S} tB $ is also syndetic (it's even thickly syndetic), $A\cap B$ is syndetic.
If $A$ is actually thickly syndetic, then, for any finite subset $F \subset G$, let $A' = \cap_{f \in F} fA$ and $B' = \cap_{f \in F} fB$. 
$A'$ and $B'$ are still thickly syndetic, so in particular $A'$ is syndetic.
Applying the above argument to $A'$ and $B'$ (instead of $A$ and $B$), one gets that 
$A' \cap B' = \cap_{f \in F} (A \cap B)$ is syndetic.
Since this holds for any $F$, $A \cap B$ is thickly syndetic.

This first steps shows that the set $T_\eps$ is syndetic (for any $\eps >0$).
Note that in order to keep with the standard definitions, syndetic, thick and thickly syndetic were defined with left-multiplication. 
Because the unique left-invariant mean on WAP functions is also right-invariant, the flight function property is also true for right multiplication (this follows from reading Glasner's book \cite[1.51 Theorem and 1.52 Lemma]{Glasner} with right multiplication instead of left multiplication).
Alternatively, just note that the random walk is symmetric: the time $n$ distributions of the left- and right-multiplications are identical. So if the random walk by left-multiplication hits the set infinitely often, then so does the random walk by right-multiplication.

The second step is to show that, if $L$ is a [right]-syndetic set, then a [irreducible] random walk [on the right] hits it almost surely. 
Indeed, there is a $S$ so that $LS=G$. 
This means that for any $g \in G$ there is a $s \in S$ and a $l \in L$ so that $g = l s $.
One can ``decorate'' the elements of $G$ with all the $s \in S$ so that $g = l s$ for some $l \in L$.
Since the set of ``decorations'' $S$ is finite, some decoration will occur infinitely many times.
Because $S$ is finite [and the random walk is irreducible], the random walk will almost surely produce all elements of $S$ starting from any decoration which occurs infinitely often.
This means that the random walks hits the set $L$ almost surely infinitely many times.
\end{proof}
This technical preparation being over, here is the analogue of Lemma \ref{tergconj-l}.
\begin{lem}\label{twmconj-l}
Assume $\pi$ is weakly mixing, $b$ is a $\mu$-harmonic cocycle,
$G$ is finitely generated and $g$ has a finite conjugacy class $C$. Then 
\[
\frac{1}{|C|}\sum_{c \in C} b(cg) = b(g). \qedhere
\]
\end{lem}
\begin{proof}
Given a function $f$, let us say that $\tilde{f}$ was obtained by $\mu$-firing $f$ if $\tilde{f} = f - a f(g) \delta_g + a f(g) \mu$ where $a \in ]0,1]$. If $h$ is $\mu$-harmonic, and $\nu$ is obtained by $\mu$-firings of $\delta_e$ then $h * \nu = h$. 

Let $F^n$ be a sequence to determine yet, but which was obtained by $\mu$ firings of $\delta_e$. Let $F^n_g$ be the translate of these functions (so that the firing began at $\delta_g$).

The proof mainly goes on as for Lemma \ref{tergconj-l}.
If $h$ is $\mu$-harmonic, then $\langle h \mid F^n_g \rangle = h(g)$.
Hence, if $C$ is a finite conjugacy class,
$\frac{1}{|C|} \sum_{c \in C} h(cg) - h(g) = \langle \nabla h \mid \tau_n \rangle$
where, again, $\tau_n$ is the transport plan obtained by taking the mass of $F^n_g$ at $g'$, split in $|C|$ masses, and take them (along a shortest path) to $g'c$ (for $c \in C$). 
For $K:= \max_{c \in C} |c|_S$ one has $\|\tau_n\|_{\ell^1} \leq K$
and 
$g'C = Cg'$, so that $\tau_n$ transports the mass uniformly to $Cg'$. 

Let $c = s_{1,c} s_{2,c} \ldots s_{|c|,c}$, then
\[
\displaystyle \tfrac{1}{|C|} \sum_{c \in C} h(cg) - h(g)  
= \tfrac{1}{|C|} \sum_{c \in C} \sum_{i = 1}^{|c|} \pgen{ \kappa_{\pi(s_1 \ldots s_{i-1}) b(s_i),\eta } \mid F^n_{g} } \\
\]
Again the left-hand side does not depend on $n$. 


As in Lemma \ref{twmsynd-l}, let $T_\eps$ be the set  where all the coefficient functions coming in the above sums are $<\eps$
Define $F^n_g$ from $F^{n-1}_g$ by firing at all the masses which are not in $T'$.
Since the set $T'$ is syndetic, a random walker will hit it almost surely. 
This implies that the mass of $F^n_g$ supported on $T^\comp$ tends to $0$ as $n \to \infty$.

Since the sums are finite and the coefficient functions bounded, the right-hand side is $\leq K|C|\eps$ as $n\to \infty$.
But this can be done for any $\eps>0$, yielding the claim.
\end{proof}
Again, a combination of the Lemmas \ref{twmtrivfc-l} and \ref{twmconj-l} yield directly
\begin{teo}\label{twmconj-t}
Assume $\pi$ is weakly mixing, $\srl{H}^1(G,\pi) \neq 0$ and $G$ is finitely generated. 
Then, for any non-trivial $\mu$-harmonic cocycle $b$, $Z^{FC}(G) \subset \ker b \cap \ker \pi_{| \img b}$.
\end{teo}
\begin{proof}
Let $z \in Z^{FC}(G)$.
Then $z$ has a finite conjugacy class $C$ (by definition).
By Lemma \ref{twmconj-l}, $\frac{1}{|C|}\sum_{c \in C} b(cg) = b(g)$ for all $g \in G$.
By Lemma \ref{twmtrivfc-l}, $b(C)$ spans a finite-dimensional subrepresentation of $\pi$ and $C \subset \ker \pi_{| \img b}$.
Since $\pi$ is weakly mixing, this subrepresentation must be $\{0\}$.
This implies that $b(C) =\{0\}$ and $C \subset \ker b$.
\end{proof}
Bekka, Pillon \& Valette \cite[\S{}4.6 and Corollary 4.13]{BPV} showed that isometric action of groups with property (FAb) also quotient through the FC-centre.

\begin{proof}[Proof of Theorem \ref{tcen-t}]
Theorem \ref{tergliou-t} or Theorem \ref{twmconj-t} basically form the statement of Theorem \ref{tcen-t} before the ``In particular''.
So any harmonic cocycle is trivial on the centre or the FC-centre (according to which point of the Theorem is under consideration).
Assume a subgroup $H$ of the centre or FC-centre is ergodic.
Theorem \ref{tergliou-t} or Theorem \ref{twmconj-t} imply that $H \subset \ker \pi_{| \img b}$.
This in turns means that $\img b =\{0\}$.
Since all harmonic cocycles have a trivial image, Corollary \ref{tidharcocredcoh-c} implies $\srl{H}^1(G,\pi)=0$.
\end{proof}

\subsection{Corollaries}

The following corollary is a new proof of a result of Guichardet \cite[Th\'eor\`eme 7 in \S{}8]{Gu}.
\begin{cor}\label{tguich-c}
Assume $G$ is nilpotent. If $\srl{H}^1(G,\pi) \neq 0$, then $1 < \pi$.
\end{cor}
\begin{proof}
Assume $1 \not<\pi$ (\ie $\pi$ is ergodic). Let $b$ be a harmonic cocycle and restrict, if necessary, to the subrepresentation $\img b$ (which is also ergodic). 
By Theorem\ref{tergliou-t} (nilpotent groups are $\mu$-Liouville for any finitely supported $\mu$; see, for example, \cite{Avez74}), one gets that $Z \subset \ker \pi$. 
Hence $b$ gives a harmonic cocycle for the quotient representation $\pi$ on $G_1 := G/Z(G)$. 
Repeat the argument on $G_1$.
Since the upper central series give the whole of $G$, one gets that $b$ must be trivial, a contradiction.
\end{proof}
Note that one could make a similar argument by replacing ``nilpotent'' by ``virtually nilpotent'' and ``$\pi$ ergodic $\implies \srl{H}^1(G,\pi)=0$'' by ``$\pi$ weakly mixing $\implies \srl{H}^1(G,\pi)=0$''. Indeed, using Theorem \ref{twmconj-t} one gets that the cocycle are trivial on the largest element of the upper FC-central series. By Duguid \& McLain \cite[Theorem 2]{DMcL}, finitely generated groups whose FC-central series end in the full group are exactly the virtually nilpotent groups. This gives 
\begin{cor}\label{twmvnil-c}
Assume $G$ is virtually nilpotent. If $\pi$ is weakly mixing, then $\srl{H}^1(G,\pi) = 0$.
\end{cor}
Shalom \cite[Corollary 5.1.3 and Lemma 4.2.2]{Shal} actually showed a stronger statement: if $\pi$ is not finite (\ie does not factor through a finite quotient of $G$) then $\srl{H}^1(G,\pi) = 0$.
\begin{cor}\label{tmimicentre-c}
Assume $G$ has an infinite FC-centre, then any representation $\pi$ with finite stabilisers has trivial reduced cohomology.
\end{cor}
\begin{proof}[First proof]
Indeed, consider $b$ a harmonic cocycle and let $\pi'$ be the subrepresentation given by restricting to the image of $b$. 
Then $\pi'$ has also finite stabilisers. 
However, by Lemma \ref{twmconj-l}, $Z^{FC}(G) \subset \ker \pi'$. 
If $Z^{FC}(G)$ is infinite and $\pi'$ has finite stabilisers, then the image of $b$ must be $\{0\}$.
\end{proof}
\begin{proof}[Second proof]
Assume $\srl{H}^1(G,\pi)\neq 0$. 
Then for some harmonic cocycle $Z^{FC}(G) \subset \ker b$, by Lemma \ref{twmconj-l}. 
By Lemma \ref{tlempt-l}, $\ker b = G$ (it cannot be finite or almost-malnormal since it contains an infinite normal subgroup).
\end{proof}
Thanks to Lemma \ref{tlempopastyle-l}, the second proof also works under the weaker hypothesis that there is some infinite subgroup $H < Z^{FC}(G)$ such that $\pi_{| H}$ has finite stabilisers.
\begin{dfn}\label{defcomp}
For a finitely generated group $G$ and a finite generating set $S$, the compression of a cocycle $b \in Z^1(G,\pi)$ is the largest increasing function $\rho_-: \zz_{\geq 0 } \to \rr_{\geq 0}$ so that $\|b(g)\|^2_\Hi \geq \rho_-(|g|_S)$. 
\end{dfn}
For example, a cocycle is proper if $\rho_-$ is unbounded. Recall that, when $G$ is an amenable group, every cocycle for the left-regular representation on $\ell^2G$ is either bounded or proper by Peterson \& Thom \cite[Theorem 2.5]{PT}.
\begin{cor}\label{tcorharmcocproperg-c}
Assume a $\mu$-harmonic cocycle $b \in \srl{H}^1(G,\pi)$ is proper, $G$ is $\mu$-Liouville and $G$ has a infinite centre. Then $\pi$ is not ergodic and $G$ surjects onto $\zz$. 
\end{cor}
\begin{proof}
By Theorem \ref{tergliou-t}, if $\pi$ is ergodic and $Z(G)$ is infinite, $\ker b$ is infinite and, hence, $b$ cannot be proper. 
So $\pi$ is not ergodic and the cocycle $b$ is not trivial when restricted to this non-ergodic subrepresentation. 
This gives directly a homomorphism to $\zz$.
\end{proof}
Likewise, one gets
\begin{cor}\label{tcorharmcocpropwm-c}
Assume a harmonic cocycle $b \in \srl{H}^1(G,\pi)$ is proper, $G$ is amenable and $G$ has a infinite FC-centre. 
Then $\pi$ is not weakly mixing and $G$ virtually surjects onto $\zz$.
\end{cor}
\begin{proof}
By Theorem \ref{twmconj-t}, if $\pi$ is weakly mixing and $G$ has a infinite FC-centre, $\ker b$ is infinite and, hence, cannot be proper. 

Restricting our attention to the finite dimensional sub-representations does not change the fact that the cocycle is harmonic. 
Hence, some finite dimensional representation of $G$ has a non-trivial reduced cohomology.
By Shalom \cite[Theorem 1.11.(1) or Theorem 4.3.1]{Shal}, this implies $G$ virtually surjects onto $\zz$.
\end{proof}
The previous corollaries imply that the compression of Liouville [resp. amenable] groups with an infinite centre [resp. FC-centre] and which do not surject [resp. virtually surject] onto $\zz$ is never realised by a harmonic cocycle.

\begin{cor}\label{tcoproptor-c}
Let $G$ be a finitely generated amenable (resp. $\mu$-Liouville) group which is torsion (resp. whose abelianisation is torsion). 
No harmonic cocycle $b$ is proper or $G$ has a finite FC-centre (resp. centre).
\end{cor}
\begin{proof}
If $G$ has an infinite FC-centre (resp. centre), Corollary \ref{tcorharmcocpropwm-c} (resp. Corollary \ref{tcorharmcocproperg-c}) implies $G$ would virtually surject (resp. surject) onto $\zz$. 
This contradicts the fact that $G$ (resp. the abelianisation of $G$) is torsion.
\end{proof}
Shalom \cite[Theorem 1.11.(1) or Theorem 4.3.1]{Shal} also shows that torsion amenable groups may not have property $H_{FD}$ (hence some weakly mixing representations $\pi$ has $\srl{H}^1(G,\pi) \neq 0$). 
The infinite dihedral group $D_\infty = \pgen{a,b \mid a^2=b^2=1}$ is an interesting example in the context of Corollary \ref{tcoproptor-c}: it is Liouville (hence amenable), it has an infinite FC-centre ($Z^{FC}(D_\infty) = \{ (ab)^n\}_{n \in \zz}$) but the centre is trivial, its abelianisation is torsion ($C_2 \times C_2$) and it has a proper harmonic cocycle (let $\pi$ be the representation on $\Hi = \rr$ which send each generator to the inversion $x \mapsto -x$ and $z$ the cocycle defined by $z(a) = -z(b) =  1$).

The authors does not know if there exists an infinite amenable torsion group which has an infinite FC-centre.

\section{$\ell^p$-cohomology}\label{slpcoh}

The aim of this section is to prove some results on the vanishing of reduced $\ell^p$ cohomology in degree one. Since most readers are probably unfamiliar with it, it seems natural to begin not only with definitions, but also with a result which shows $\ell^p$-cohomology has implications on Hilbertian representations.

\subsection{Preliminaries and applications to unitary representations}\label{sslprem}

For a group $H$, let $\lmd_{\ell^p H}$ denote the left-regular representation on $\ell^p H$, \ie the representation coming from the action of $H$ on $X = H$. The associated reduced cohomology is called the reduced $\ell^p$-cohomology. 

A very nice application of reduced $\ell^p$-cohomology in degree one to questions of sphere packings may be found in Benjamini \& Schramm \cite{BS}. Other important applications include problems of quasi-isometries (see Pansu \cite{Pan-rs}), the conformal boundary of hyperbolic spaces (see Bourdon \& Pajot \cite{BP} and Bourdon \& Kleiner \cite{BK}), the critical exponent for some actions (see Bourdon, Martin \& Valette \cite{BMV}), nonlinear potential theory (see Puls \cite{Puls-pharm} and Troyanov \cite{Tro}) and existence of harmonic functions with gradient conditions (see \cite[Theorem 1.2 or Corollary 3.14]{Go} or \cite{Go-cras}). 

For the reader whose interest lies mostly in Hilbert spaces, here is a reason to consider reduced $\ell^p$-cohomology. 
The following result is implicitly mentioned in \cite[\S{}2]{GJ}. 
Recall that a function $f:G \to \cc$ is said to be \textbf{constant at infinity} if it belongs to the linear span of $c_0G$ and the constant function\footnote{ 
In other words, there exists a $c \in \cc$ such that $\forall \eps >0$ the set $G \setminus f^{-1}(B_\eps(c))$ is finite (where $B_\eps(c) = \{ k \in \cc \mid |c-k|<\eps\}$).}.
\begin{cor}\label{tcoefflpharm-c}
Assume $\srl{H}^1(G,\ell^pG) = 0$ for some $p>1$ and $G$ is finitely generated. 
Then, for any $0 < p' \leq p$ and for any unitary representation $\pi$ with finitarily coefficients in $\ell^{p'}$, $\srl{H}^1(G,\pi) =0$.
\end{cor}
\begin{proof}
By Guichardet \cite[Th\'eor\`eme 7 in \S{}8]{Gu} (see also Corollary \ref{tguich-c}), one may assume $G$ is not nilpotent (such groups would otherwise need to be considered, see Remark \ref{rreflpcoh}.1). 
Let $b \in \srl{H}^1(G,\pi)$ be a harmonic cocycle.
If $b$ is not trivial, Lemma \ref{tgradcond-l} and Remark \ref{rvirtcobharm} imply there is a non-constant harmonic function $h_\eta:G \to \cc$ which has gradient in $\ell^{p'}$. 
By \cite[Theorem 1.2 or Corollary 3.14]{Go}, 
the existence of such a functions implies that $\srl{H}^1(G,\ell^{p'}G) \neq 0$ and this, in turn, implies that $\srl{H}^1(G,\ell^pG) \neq 0$; a contradiction.
\end{proof}
This will allow us to show that, for many groups, representations with finitarily coefficients in $\ell^p$ have trivial reduced cohomology. Note that the converse of Corollary \ref{tcoefflpharm-c} is false: \.Zuk \cite[Theorem 3 and 4]{Zuk} showed there are hyperbolic groups (so $\srl{H}^1(G,\lmd_{\ell^pG}) \neq 0$ for some $p$, see \eg \cite{BP}) with property (T) (so $H^1(G,\pi)=0$ for any unitary representation).

\begin{rmk}\label{rreflpcoh}
It is known that the reduced $\ell^p$-cohomology is trivial in degree 1 for the following groups ($1<p<\infty$):
\begin{enumerate}\renewcommand{\itemsep}{-1ex} \renewcommand{\labelenumi}{\bf \arabic{enumi}.}
 \item $G$ has an infinite FC-centre (see Kappos \cite[Theorem 6.4]{Kappos}, Martin \& Valette \cite[Theorem 4.3]{MV}, Puls \cite[Theorem 5.3]{Puls-Can}, Tessera \cite[Proposition 3]{Tes} or \cite[Theorem 3.2]{Go-trans})
 \item $G$ has a finitely supported measure with the Liouville property, \ie no bounded $\mu$-harmonic functions (see \cite[Theorem 1.2 or Corollary 3.14]{Go}). This includes all polycyclic groups (for such groups, see also Tessera \cite{Tes})
 \item $G$ is a direct product of two infinite finitely generated groups (see \cite[Corollary 3]{Go-cras}). 
 \item $G$ is a wreath product with infinite base group (see \cite[Proposition 1]{Go-cras} and Martin \& Valette \cite[Theorem.(iv)]{MV}) unless the base group has infinitely many ends and the lamp group is amenable. Arguments from Georgakopoulos \cite{Geo} show that this also holds for finite lamp groups (even if the base group has infinitely many ends).
 \item $G$ is some specific type of semi-direct product $N \rtimes H$ with $N$ not finitely generated (see \cite{Go-lpharm} for the full hypothesis).
\end{enumerate}
It is also trivial in any amenable group for any $1<p \leq 2$ (see \cite{Go}).

On the other hand it is non-trivial (for some $p <+\infty$) in all hyperbolic groups (see Bourdon \cite{Bourdon-hyp}, Bourdon \& Pajot \cite{BP}, \'Elek \cite{Elek}, Gromov \cite[p.258]{Gro}, Pansu \cite{Pan} or Puls \cite[Corollary 1.4]{Puls-floyd}), some groups without free subgroups of rank $2$ and some torsion groups of infinite exponent (see Osin \cite{Osin-tor}).

The reduced $\ell^1$-cohomology in degree one is non-trivial if and only if the group has $\geq 2$ ends (see \cite[Appendix A]{Go}).
\end{rmk}
Before moving on, let us note that in Corollary \ref{tcoefflpharm-c}, $G$ needs not necessarily be finitely generated.
Indeed, if $G$ is not finitely generated, it suffices to prove the statement for any finitely generated subgroup. This is due to the following lemma (this version is taken from Martin \& Valette \cite[Lemma 2.5]{MV}):
\begin{lem}
Assume $G$ is a countable group given as a union of finitely generated groups $G = \cup_i G_i$.
Let $b \in Z^1(G,\pi)$ for some representation $\pi$ of $G$ on the Banach space $\BB$. Then:
$\forall i, b_{|G_i} \in \srl{B}^1(G_i,\pi)$ 
implies $b \in \srl{B}^1(G,\pi)$
which in turn implies that for all $G'<G$ finitely generated, $b_{|G'} \in \srl{B}^1(G,\pi)$.
\end{lem}
%
Combined with Martin \& Valette \cite[Proposition 2.6]{MV}, one gets ($G$ is here always assumed countable)
\begin{prop}\label{tcorMV-c}
The following are equivalent:
\begin{itemize}\renewcommand{\itemsep}{-1ex}
 \item $\srl{H}^1(G,\lmd_{\ell^pG})=0$.
 \item For any finitely generated $H<G$, $\srl{H}^1(H,\lmd_{\ell^pH})=0$. 
 \item For some increasing sequence of finitely generated subgroups $\{H_i\}$, $G =  \cup_i H_i$ and $\srl{H}^1(H_i,\lmd_{\ell^pH_i})=0$.
\end{itemize}
\end{prop}
When one considers the $\ell^p$-cohomology, the space of virtual coboundaries is also called the space of $p$-Dirichlet functions $\D^pG := \D^p(G,\lmd_{\ell^pG})$, \ie the space of functions $f:G \to \cc$ so that $\nabla f \in \ell^pE$ (where $E$ are the edges of $\cayl(G,S)$). 
The norm of a cocycle (as introduced in \S{}\ref{ssbanachrep}) is the $\ell^p$-norm of $\nabla f$.
Henceforth, it will be referred to as the $\D^pG$-norm.
See Martin \& Valette \cite[\S{}3]{MV} or Puls \cite{Puls-Can,Puls-harm}  for more background.

It is fairly classical that the coboundary map $d:\ell^pG \to B^1(G,\lmd_{\ell^pG})$ has closed image if and only if the group is not amenable (see Guichardet \cite[Th\'eor\`eme 1 and Corollaire 1]{Gu}). 
In particular, $H^1(G,\lmd_{\ell^pG})=0$ implies $G$ is non-amenable.

\subsection{Triviality and values at infinity}\label{ssbndval}

Let us now present an improvement of a result of \cite[Lemmas 3.1 and 3.9]{Go} showing that (under a growth hypothesis) functions in $\D^pG$ corresponding to the trivial class are exactly those which  are constant at infinity. The improvement is not major (relaxes the hypothesis on growth) but it makes for a good opportunity to present this important ingredient of the upcoming proof.
Some concepts from nonlinear potential theory will also come in handy.
\begin{dfn}
Let $(X,E)$ be an infinite connected graph. The \textbf{inverse $p$-capacity}\footnote{One might also like to call it the ``$p$-resistance to $\infty$''.} of a vertex $x \in X$ is 
\[
\cpc_p(x) := \big( \inf \{ \| \nabla f\|_{\ell^pE} \mid f:X \to \cc \text{ is finitely supported and } f(x) = 1 \} \big)^{-1}
\]
The graph is called $p$\textbf{-parabolic} if $\cpc_p(x) =+\infty$ for some $x \in X$. 
A graph is called $p$-\textbf{hyperbolic} if it is not $p$-parabolic.
\end{dfn}
Recall (see Holopainen \cite{Holo}, Puls \cite{Puls-pharm} or Yamasaki \cite{Yama}) that if $\cpc_p(x_0) =0$ for some $x_0$ then $\cpc_p(x) =0$ for all $x \in X$. Recall also that $2$-parabolicity is equivalent to recurrence.
\begin{rmk}\label{rKLSW}
~\\[-1.9em]
\begin{enumerate}\renewcommand{\itemsep}{-1ex} \renewcommand{\labelenumi}{\bf \arabic{enumi}.}
\item If the graph $(X,E)$ is vertex-transitive, $\cpc_p(x) = \cpc_p(y)$ for all $x,y \in X$. 
Let $\cpc_p (X) := \cpc_p(x)$ be this constant. 
It is also easy to see that if the automorphism group acts co-compactly on the graph, the inverse $p$-capacity is bounded from below.

\item Note that in the definition of $p$-capacity, one may also assume that the functions take value only in $\rr_{\geq 0}$. 
Indeed, looking at $|f|$ instead of $f$ reduces the norm of the gradient. 
Likewise, one can even assume $f$ takes value only in $[0,1]$ as truncating $f$ at values larger than $1$ will again reduce the norm of the gradient. \qedhere
\end{enumerate}
\end{rmk}
The following proposition is an adaptation of a result of Keller, Lenz, Schmidt \& Wojchiechowski \cite[Theorem 2.1]{KLSW}. 
\begin{prop}
Assume $G$ is a finitely generated group with growth at least polynomial of degree $d$ and $p<d$. If $f \in \D^pG$ represents a trivial class in reduced $\ell^p$-cohomology, then $f$ is constant at infinity.

Furthermore, $c_0G \subset \D^pG$ and $\forall f \in c_0G,  \|f\|_\infty \leq \cpc_p(G,S) \|\nabla f\|_p$.
\end{prop}
\begin{proof}
A consequence of the Sobolev embedding corresponding to $d$-dimensional isoperimetry is that groups of growth of at least polynomial of degree $d >p$ have $p$-hyperbolic Cayley graphs. See Troyanov \cite[\S{}7]{Tro} as well as Woess' book \cite[\S{}4 and \S{}14]{Woe} and references therein for details.

As $\cayl(G,S)$ is $p$-hyperbolic and by Remark \ref{rKLSW}.2, one has
$
\forall f \text{ of finite support } |f(x)| \leq \cpc_p(x) \|\nabla f\|_p.
$
However, by Remark \ref{rKLSW}.1, there is no dependence on $x$ on the right-hand side.
So 
$
\forall x \in G, \forall f \text{ of finite support } |f(x)| \leq \cpc_p \|\nabla f\|_p
$
where $\cpc_p$ is $\cpc_p \big(\cayl(G,S) \big)$ 
Trivially this implies 
\[
\forall f:G \to \cc \text{ of finite support } \|f\|_\infty \leq  \cpc_p \|\nabla f\|_p 
\]
As a first consequence, 
assume $f_n \overset{\D^p}\to f $ with $f_n$ finitely supported. 
Then $f_n$ also converge to $f$ in $\ell^\infty G$. 
Since $c_0G$ is the closure of finitely supported functions in $\ell^\infty G$, this shows that $f \in \srl{\ell^pG}^{\D^p}$ implies $f \in c_0G$.
In other words, if $f$ represents a trivial class in reduced $\ell^p$-cohomology, then $f$ is constant at infinity.

As a second consequence, let us show the ``Furthermore''. 
Pick some $f \in c_0G$. Apply the inequality to $g_\eps = f-f_\eps$ where $f_\eps$ is the truncation of $f$:
\[
f_\eps(x) = \left\{\begin{array}{ll}
\eps f(x)/|f(x)| & \text{if } |f(x)| > \eps\\
f(x) & \text{ else.}
\end{array}\right.
\]
Indeed, $g_\eps$ is finitely supported so it satisfies $\|g_\eps\|_\infty \leq \cpc_p \|\nabla g_\eps\|_p$ (recall that $\cpc_p = \cpc_p\big(\cayl(G,S)\big)$).
Also $\|\nabla g_\eps\|_p \leq \|\nabla f\|_p$ and $\|f\|_\infty \leq \eps+ \|g_\eps\|_\infty$. 
Hence $\|f\|_\infty \leq \eps + \cpc_p\|\nabla f\|_p$ and the conclusion follows by letting $\eps \to 0$.
\end{proof}
The above proposition gives the following very nice characterisation of virtual cocycles corresponding to the trivial class. It improves a result of \cite[Lemmas 3.1 and 3.9]{Go} by requiring only $d>p$ instead of $d >2p$.
\begin{cor}\label{tbndval-c}
Assume $G$ is a finitely generated group with growth at least polynomial of degree $d$ and assume $d>p$. $f \in \D^pG$ represents a trivial class in reduced $\ell^p$-cohomology if and only if $f$ is constant at infinity.
\end{cor}
\begin{proof}[Quick proof]
Here is a quick version of the proof (which can be found as \cite[Lemma 3.1]{Go}). 
Without loss of generality the constant at infinity is $0$ (because one may add a constant function to $f$).
Considering again $g_\eps = f-f_\eps$ (where $f_\eps$ is the truncation of $f$), one can check that, as $\eps \to 0$,  $g_\eps \overset{\D^p}\to f$. 
Since $g_\eps$ is finitely supported, it is in $\ell^pG$ (and this concludes the proof).
\end{proof}

As in Keller, Lenz, Schmidt \& Wojchiechowski \cite{KLSW}, say that the graph $(X,E)$ is \textbf{uniformly $p$-hyperbolic} if $\cpc_p(X,E) := \supp{x \in X} \cpc_p(x)$ is finite. 
Using the arguments from \cite{Go}, one can show:
\begin{lem}\label{tunifpar-l}
If $(X,E)$ is a graph of bounded valency with $d$-dimensional isoperimetry and $d>2p$, then $(X,E)$ is uniformly $p$-hyperbolic.
\end{lem}
\begin{proof}
First, recall that $d$-dimensional isoperimetry implies that the Green's kernel ($k_o:= \sum_{n \geq 0} P^n_o$ where $P^n_o$ is the random walk distribution at times $n$ starting at the vertex $o$) has an $\ell^q X$-norm (for some $q < p'= \frac{p}{p-1}$) which is bounded independently from $o$.

Indeed, $d$-dimensional isoperimetry implies that $\| P^n_o\|_\infty \leq \kappa n^{-d/2}$ (where $\kappa \in \rr$ comes from the constant in the isoperimetric profile; see Woess' book \cite[(14.5) Corollary]{Woe} for details).
From there, one gets that $\| P^n_o \|_q^q \leq \|P^n_o\|_q^{q-1} \|P^n_o\|_1 \leq \kappa^{q-1} n^{-d(q-1)/2}$.
This implies that $\| k_o\|_q \leq \sum_{n \geq 0} \kappa^{1/q'} n^{-d/2q'}$ (a series which converges if $d > 2q'$).

Second, let $f$ be a finitely supported function with $f(o)=1$, then 
\[
\pgen{ \nabla f \mid \nabla k_o } = \pgen{ f \mid \nabla^* \nabla k_o} = \pgen{ f \mid \delta_o} = f(o)=1.
\]
Since $\| \nabla f \|_p \geq \|\nabla k_o\|_{p'}^{-1} \pgen{ \nabla f \mid \nabla k_o }$, $\| \nabla k_o\|_{p'} \leq 2\nu \|k_o\|_{p'} \leq 2 \nu \|k_o\|_q  = 2 \nu \kappa_q^{-1}$  (where $\nu$ is the maximal valency of a vertex) and there is no dependence in $o$, this means that $\cpc_p(X,E) \leq  \kappa_q /2\nu$.

Noting that, for the above, the conditions $q \leq p'$ and $2q'< d$ need to hold, one gets that the bound holds as long as $2p < d$.
\end{proof}
Troyanov \cite[\S{}7]{Tro} defines a parabolic and isoperimetric dimensions and shows the inequality $d_{par} \geq d_{isop}$.
Actually, if one defines analogously the ``uniform parabolicity dimension'' $d_{u.p.} := \inf \{ p \mid (X,E)$ is uniformly $p$-parabolic$\}$, Lemma \ref{tunifpar-l} shows that $d_{u.p.} \geq \tfrac{1}{2} d_{isop}$. 
A more obvious inequality is $d_{par} \geq d_{u.p.}$.

One also gets the following corollary on the $p$-harmonic boundary and the $p$-Royden boundary.
For the definitions see Puls \cite[\S{}2.1]{Puls-pharm}.
\begin{cor}\label{tcorroyd-c}
Assume 
\begin{itemize}\renewcommand{\itemsep}{-1ex}
 \item either that $p<d$ and $\Gamma$ is the Cayley graph of a group $G$ which is finitely generated and has growth at least polynomial of degree $d$;
 \item or that $p< d/2$ and $\Gamma$ is a graph with a $d$-dimensional isoperimetric profile.
\end{itemize}
Then the $p$-Royden boundary and the $p$-harmonic boundary are equal.
\end{cor}
The proof is identical to that of the analogous result for $p=2$ by Keller, Lenz, Schmidt \& Wojchiechowski; see \cite[Theorem 4.1]{KLSW}. 

Combined with Puls \cite[Theorem 2.5]{Puls-pharm}, this shows that many groups have only one point in the $p$-Royden boundary (\ie any group which is not nilpotent and has trivial reduced $\ell^p$-cohomology; see Remark \ref{rreflpcoh} and the upcoming subsections for examples).

\subsection{Applications to $\ell^p$-cohomology}\label{sslpqnor}

A special case of Theorem \ref{tteoptweak-t} applied to the left-regular representation on $\ell^pG$ (which is strongly mixing) combined with a result from Martin \& Valette \cite[Proposition 2.6]{MV} (see Proposition \ref{tcorMV-c} above), one gets a refinement of Bourdon, Martin \& Valette \cite[Theorem 1.1)]{BMV}:
\begin{cor}\label{tcorptlp-c}
Assume $K<G$ is an infinite subgroup and $H^1(K,\lmd_{\ell^pK})=\{0\}$.
Then either $H^1(G,\lmd_{\ell^pG}) =0$ or there is an almost-malnormal subgroup $H \lneq G$ so that $K < H$. In particular, if $K$ is wq-normal then $H^1(G,\lmd_{\ell^pG}) =0$
\end{cor}
However, using Corollary \ref{tbndval-c}, one can prove an essentially finer result.
The following theorem is not only a generalisation of 
Bourdon, Martin \& Valette \cite[Theorem 1.1)]{BMV} (if $N \lhd G$ is infinite and $N<H<G$ then $H$ is q-normal), but also of \cite[Theorem 1.4]{Go}. 
Except for the finite generation hypothesis, it is also stronger than Corollary \ref{tcorptlp-c}.
\begin{teo}\label{tqnorlpcoh-t}
Let $1 < p < d \in \rr$. Assume $K$ is wq-normal in $G$, both $K$ and $G$ are finitely generated, $K$ has growth at least degree $d$ polynomial and $\srl{H}^1(K,\lmd_{\ell^p K})=0$. Then $\srl{H}^1(G,\lmd_{\ell^p G}) =0$. 
\end{teo}
\begin{proof}
Let $S$ be a finitely generating set of $G$ for which $K$ is q-normal. 
Without loss of generality, one may assume $S$ contains a finite generating set for $K$ called $S_K$.
Let $b \in Z^1(G,\lmd_{\ell^p G})$ be a cocycle and write it as a virtual coboundary: $b = \pi(g) f -f $ where $f \in \D^pG$ (\ie $f$ is a function on $\cayl(G,S)$ with gradient in $\ell^p$). 
Decompose $G = \sqcup_i  K g_i$ into $K$-cosets. 
The graph restricted to any of these cosets is isomorphic to $\cayl(K,S_K)$ (the map is $kg_i \mapsto k$).
Furthermore, \emph{via} this isomorphism, the function $f$ restricted to any coset is a function in $\D^pK$. 

Using Corollary \ref{tbndval-c} (since $\srl{H}^1(K,\lmd_{\ell^p K})=0$ and $K$ has growth at least $d$), $f$ takes only one value at infinity on each subgraphs $K g_i$. 
Since $f \in \D^pG$, the following sum is finite:
\[
\| \pi(g)f-f \|_{\ell^p}^p 
= \sum_{g_i \in K \bsl G, k \in K} | f( g^{-1}kg_i) - f(kg_i)|
= \sum_{g_i \in K \bsl G, k \in K} | f( g^{-1}kg g^{-1}g_i) - f(kg_i)|
\]
If $g$ is [non-trivial and] in the set generating the quasi-normaliser of $K$, $N^q_G(K)$, there are infinitely many $k \in gKg^{-1} \cap K$. 
In particular, there is a sequence $k_n$ so that $ g^{-1}k_n g g^{-1}g_i$ tend to infinity in $K g_j$ (where $Kg_j = Kg^{-1}g_i$) and  $k_ng_i$ tens to infinity in $K g_i$. 
This implies that the constants at infinity on these cosets must be the same, otherwise the sum diverges and $f \notin \D^pG$. 
This shows that the constant at infinity is the same on each $K$ coset which lie in the same $N^q(K)$ coset.
However, the argument can be reapeated on $N^{q,2}(K) := N^q_G(N^q_G(K))$.

By pursuing this using transfinite induction, one gets that constant is the same on $N^{q,\alpha}(K)= G$
By Corollary \ref{tbndval-c}, one sees that the class of $f$ is the trivial class.
\end{proof}
Note the conclusion also follows from a maximum principle for $p$-harmonic functions (see \S{}\ref{ssbanachrep} and \S{}\ref{ssvirtcob} or Puls \cite{Puls-Can,Puls-harm} or  Martin \& Valette \cite[\S{}3]{MV}). 
Bourdon in \cite[\P{}\textbf{4)} in \S{}1.6]{Bourdon} (see also \cite[Example 1 in \S{}3]{BMV}) has given a very nice example showing that the hypothesis that $\srl{H}^1(K,\lmd_{\ell^pK})=0$ cannot be dropped from Theorem \ref{tqnorlpcoh-t}.

Next, let us show an improvement of Martin \& Valette \cite[Theorem 4.2]{MV}. 
The main interest in the following proof is that the vanishing of the cohomology for the subgroup is no longer necessary. 
This comes at the cost of an hypothesis on the FC-centraliser.

Two functions $f_1,f_2 : G \to \cc$ will be said to have the \textbf{same value at infinity}\footnote{
In other words, for any sequence $g_n \to \infty$ (\ie $g_n$ exits any finite set), $|f_1(g_n)-f_2(g_n) | \to 0$.}
if $f_1-f_2$ belongs to $c_0G$.
\begin{teo}\label{tqnorcen-t}
Assume $G$ is finitely generated. Let $K<G$ be a finitely generated wq-normal subgroup with growth at least polynomial of degree $d$. Assume its FC-centraliser $Z_G^{FC}(K)$ is infinite and that $p<d$. 
Then $\srl{H}^1(G,\lmd_{\ell^pG}) =0$ .
\end{teo}
\begin{proof}
If $Z_G^{FC}(K) \cap K = Z_K^{FC}(K)$ is infinite, then $K$ has an infinite FC-centre. 
Remark \ref{rreflpcoh}.1 and Theorem \ref{tqnorlpcoh-t} give the claim.

So assume $Z_K^{FC}(K)$ is finite and $Z_G^{FC}(K)$ is infinite. 
Many elements of the rest of this proof resembles the proof of Theorem \ref{tqnorlpcoh-t}.
Let $f \in \D^pG$ be a virtual coboundary for a cocycle $b$ and let $G = \sqcup_i K g_i$.
Consider further $z_i \in Z_G^{FC}(K)$ so that $g_i= z_i g_{j(i)}$. 
A first useful claim is that, since $K \cap Z_G^{FC}(K)$ is finite and $Z_G^{FC}(K)$ is infinite, given $i_1$, there are infinitely many $i_2$ with $j(i_2) = j(i_1)$.

To prove this claim, recall that given two subgroups $K,L <G$, $[L : L \cap K] = [KL : K]$ ($KL$ is not necessarily a subgroup, but it can nevertheless be split into $K$-(left)-cosets ). 
This follows from the orbit-stabiliser theorem: $L$ acts transitively (on the right) on the $K$-cosets in $KL$ and the element $K$ has stabiliser $K \cap L$. 
Taking here $L = Z_G^{FC}(K)$, this means that there are also infinitely many $K$-cosets in $KL$ because $K \cap L$ is finite and $L$ is infinite,

As in Theorem \ref{tqnorlpcoh-t}, let $f_i$ be the restriction of $f$ to the coset $Kg_i = Kz_ig_{(i)}$; $f_i$ is identified to an element of $\D^pK$.
If for all $i$, $f_i$ is trivial in reduced cohomology, then the proof is identical to that of Theorem \ref{tqnorlpcoh-t} (the hypothesis $\srl{H^1}(K,\lmd_{\ell^pK})=0$ was used to this effect).
This means one can assume that for some $i_0$, $f_{i_0}$ is not constant at infinity.

To $f_{i_0}$ one can associate a $p$-harmonic function $h$, which is the element of $\D^pK$ with minimal norm which takes the same values at infinity as $f_{i_0}$ (again, see Martin \& Valette \cite[\S{}3]{MV}, Puls \cite{Puls-Can,Puls-harm} or \S{}\ref{ssbanachrep} and \S{}\ref{ssvirtcob}).
Let $i'$ be such that $j(i') =j(i_0) =:j_0$.
The distance from $k z_{i'} g_{j_0} = k z_{i'} k^{-1} k g_{j_0}$ to $k z_i g_{j_0} = k z_ik^{-1}k g_{j_0}$ is bounded by $\maxx{k \in K} |kz_i^{-1}z_{i'}k^{-1}|$ . 
Because the gradient of $f$ is in $\ell^p$, this implies that, for any $i'$ with $j(i') = j(i_0)$, $f_{i_0}$ and $f_{i'}$ take the same values at infinity.

The $\D^pK$-norm of the $f_{i'}$ is however uniformly bounded from below by the $\D^pK$-norm of $h$ (as $h$ has the smallest $\D^p$ among all functions which take the same value at infinity as $f_{i_0}$). 
But there are infinitely many such restrictions,
and the $\D^pG$-norm of $f$ includes the sum of all these $\D^pK$-norms.
So $f$ has infinite $\D^pG$-norm, a contradiction.

Thus for any $i$, $f_i$ is constant at infinity.
This means that $f$ takes only one value at infinity on each subgraph $Kg_i$. 
From there on, the proof is identical to that of Theorem \ref{tqnorlpcoh-t} (the hypothesis $\srl{H^1}(K,\lmd_{\ell^pK})=0$ was used to this effect).
\end{proof}
The proof in the case where $Z_K^{FC}(K)$ is finite can be shortened significantly if $Z_G(K)$ contains an infinite finitely generated subgroup $Z'$ 
(\ie $Z_G(K)$ is not locally finite).
Indeed, the subgroup generated by $K$ and $Z'$ is then isomorphic to a direct product (and is still q-normal).
The claim then follows by Remark \ref{rreflpcoh}.3 and Theorem \ref{tqnorlpcoh-t}.


\subsection{Further corollaries}\label{sslpcor}

Before moving on to a larger class of groups, let us make a simple example.
\begin{exa}\label{exBS}
Let $G = \pgen{a,b \mid ba^pb^{-1} = a^q}$ (with $p,q \in \zz^\times$) be the Baumslag-Solitar group. 
Let $K = \pgen{ a, bab^{-1}}$. 
Then $K \simeq \pgen{ a,y \mid y^p = a^q}$ (by the isomorphism $y := bab^{-1}$) has exponential growth as soon as $|p| \neq 1 \neq |q|$ (because it surjects on $\pgen{a,y\mid y^p=a^q=1} \simeq C_p * C_q$) and is q-normal in $G$ for the generating set $\{a,b\}$. 
On the other hand $K$ has an infinite centre (the subgroup generated by $a^q$), hence $\srl{H}^1(K,\lmd_{\ell^p K})=0$ (see Remark \ref{rreflpcoh}.1).
By Theorem \ref{tqnorlpcoh-t}, $\srl{H}^1(G,\lmd_{\ell^pG})=0$.

Note that if $|p| = |q|$, $G$ has an infinite centre so that the conclusion follows directly from Remark \ref{rreflpcoh}.1. 
Also, the solvable Baumslag-Solitar (\ie when $|p|=1$ or $|q|=1$) groups are already known to have $\srl{H}^1(G,\lmd_{\ell^p G})=0$ by Remark \ref{rreflpcoh}.2.
\end{exa}
\begin{exa}\label{exHig}
Let $G = \pgen{a_i , i \in \zz / 4\zz \mid a_{i+1}a_ia_{i+1}^{-1} = a_i^2}$ be the Higman 4-generator 4-relator group.
Let $H_3$ be the subgroup generated by the $a_0, a_1$ and $a_2$ and let $H_2$ be the one generated by $a_0$ and $a_1$.
$H_2$ is isomorphic to a Baumslag-Solitar group (with $p=1$ and $q=2$). 
In particular, $\srl{H}^1(H_2,\lmd_{\ell^p H_2})=0$ (by the previous example).
$H_2$ is q-normal in $H_3$: $a_2 H_2 a_2^{-1} \cap H_2 \supset \{ a_1^{2n} \}_{n \in \zz}$.
$H_3$ is also q-normal in $G$: $a_3 H_3 a_3^{-1} \cap H_3 \supset \{ a_2^{2n} \}_{n \in \zz}$.
By Theorem \ref{tqnorlpcoh-t}, $\srl{H^1}(G,\lmd_{\ell^p G})=0$.
%
\end{exa}

The previous examples illustrates an important gain made by considering q-normality. There are finitely generated groups (such as $\zz \wr \zz$) where the [non-trivial] normal subgroups are not finitely generated. In such a group \cite[Theorem 1.4]{Go} cannot be applied. However, any infinite finitely generated subgroup $K_0$ of $H$ gives rise to a q-normal subgroup: look at $K$ the subgroup generated by $\cup_{g \in S_G} g K_0 g^{-1}$ where $S_G$ is a generating set of $G$. This means there are lots of candidates to apply Theorem \ref{tqnorlpcoh-t}.

Solvable groups have ``few'' malnormal subgroups (and usually ``many'' subnormal subgroups) so they make natural examples for the application of Theorem \ref{tqnorlpcoh-t}.
Recall that derived series of a group $H$ is defined by $H^{(i+1)} = [H^{(i)},H^{(i)}]$ with $H^{(0)} = H$.
The derived length of a solvable group $G$ (\ie a group for which the derived series stabilises at $\{1\}$ after finitely many steps) is the smallest $k$ such that $G^{(k)}$ is trivial.
The Hirsch length (the number of infinite cycle factors in the quotients $G^{(i+1)}/G^{(i)}$) may be infinite even if the derived length is finite.
\begin{exa}\label{exFM}
The free solvable group of derived length $k$ and rank $n$, \ie $G \simeq F_n / F_n^{(k)}$ (where $F_n$ is the free group on $n$ generators),
are groups to which Theorem \ref{tqnorlpcoh-t} applies. 
Indeed, for any $d$, $F_n^{(k-1)}/F_n^{(k)}$ contains subgroups isomorphic to $\zz^d$ which are q-normal.
\end{exa}
More generally, 
\begin{cor}\label{tsolvlp-c}
Assume $G$ is a finitely generated solvable group of derived length $k$ and $\mr{rank}_\zz G^{(k-1)} \geq d$, then, for $p<d$, $\srl{H}^1(G,\lmd_{\ell^p G})=0$.
\end{cor}
[The case $p=1$ is slightly singular and need not be addressed here (see Remark \ref{rreflpcoh}).]
\begin{proof}
The characteristic subgroup $G^{(k-1)}$ is Abelian. 
Take a subgroup $K'< G^{(k-1)}$ isomorphic to $\zz^d$. 
Let $S$ be a generating set for $G$, $S'$ a finite generating set for $K'$ and $K$ be the group generated by $\cup_{g \in S} g S'g^{-1}$.
$K$ now satisfies all the hypothesis of Theorem \ref{tqnorlpcoh-t} ($K<G^{(k-1)}$ being Abelian, its reduced cohomology vanishes; $K$ grows faster than $K'$; $K$ is a finitely generated q-normal subgroup of $G$) and the conclusion follows. 
 \end{proof}
It could, of course, happen that $G^{(k-1)}$ is locally finite, in which case one could try to apply Theorem \ref{tqnorlpcoh-t} on $G^{(k-2)}$ or any subgroup containing $G^{(k-1)}$. If $G^{(k-1)}$ is finitely generated, then the finitely generated subgroups of $G^{(k-2)}$ are prime candidates (at worse they are nilpotent and at best polycyclic; see Remark \ref{rreflpcoh}.1-2)

The hypothesis of finite generation from Theorems \ref{tqnorlpcoh-t} and \ref{tqnorcen-t} may be dropped if one requires more normality in the ascending sequence.
\begin{cor}\label{tasclp-c}
Let $\Gm$ be such that
\begin{enumerate}\renewcommand{\itemsep}{-1ex} \renewcommand{\labelenumi}{\bf \arabic{enumi}.}
\item there is a subgroup $H$ which is wq-normal with respect to the sequence $\{H_\alpha\}_{\alpha \leq \beta}$,
\item there is a subgroup $K$ of $H$ which is finitely generated and has growth at least polynomial of degree $d>p$,
\item the inclusion $H_\alpha < H_{\alpha+1}$ is q-normal when $H_\alpha$ is finitely generated and normal otherwise,
\item either $\srl{H}^1(H, \lmd_{\ell^p H} ) =\{0\}$ or $Z^{FC}_{H_1}(H)$ is infinite.
\end{enumerate}
Then $\srl{H}^1(\Gm,\lmd_{\ell^p\Gm}) = \{0\}$.
\end{cor}
\begin{proof}
If $H_0:=H$ is finitely generated, either Theorem \ref{tqnorlpcoh-t} or Theorem \ref{tqnorcen-t} apply directly. 
So, by hypothesis, it may be assumed that $H_0 \lhd H_1$.

Let $S$ be a finite generating set for $K$.
For any finitely generated subgroup $K_1 < H_1$ with generating set $S_1$, let $K'$ be the group generated by $\cup_{g \in S_1} g S g^{-1}$.
Since $H_0 \lhd H_1$, $K' < H_0$.
If $\srl{H}^1(H, \lmd_{\ell^p H} ) =\{0\}$, then, by Corollary \ref{tcorMV-c}, any finitely generated $K' < H_0$ will satisfy $\srl{H}^1(K', \lmd_{\ell^p K'} ) =\{0\}$.
Hence Theorem \ref{tqnorlpcoh-t} can be applied to $K'$ (it is q-normal in $K_1$ and grows faster than $K$).
If $Z^{FC}_{H_1}(H)$ is infinite, then so is $Z^{FC}_{H_1}(K')$ (as $Z^{FC}_{H_1}(K') \supset Z^{FC}_{H_1}(H)$).
Apply Theorem \ref{tqnorcen-t} to $K'$ to conclude that $\srl{H^1}(K',\lmd_{\ell^p}) =0$.

This shows any finitely generated $K_1 < H_1$ satisfies $\srl{H}^1(K_1, \lmd_{\ell^p K_1} ) =\{0\}$. 
The conclusion passes to $H_1$ by Corollary \ref{tcorMV-c}.
Transfinite induction (using Corollary \ref{tcorMV-c} again at the inaccessible ordinals) gives the conclusion.
\end{proof}
The methods of the previous corollary could be use to cover many other groups, but these do not seem to fit in any nicely described class. 
Many hyperabelian groups are covered by this corollary.
For example, there are finitely generated hyperabelian non-solvable groups (see Hall \cite[\P{}2 of p.539 in \S{}1.7]{Hall}) to which Corollary \ref{tasclp-c} applies.

%
%
%

\section{Questions }\label{sques}

Here is a conjecture motivated by Osin \cite[Problem 3.3]{Osin-ques} (does $\srl{H}^1(\Gm,\lmd_{\ell^2\Gm}) \neq 0$ and finite presentation implies acylindrically hyperbolic) and Gromov \cite[\S{}8.$A_1$.($A_2$), p.226]{Gro} (does $\Gm$ amenable implies $\srl{H}^1(\Gm,\lmd_{\ell^p\Gm}) = 0$).
\begin{conj}
Assume $\Gm$ is a torsion-free finitely presented group. If, for some $p \in ]1,\infty[$, $\srl{H}^1(\Gm,\lmd_{\ell^p\Gm}) \neq 0$ then $\Gm$ contains a free subgroup.
\end{conj}
One could also strengthen the hypothesis to ``finite $K(\Gm,1)$''.
It would be nice to construct the free subgroup by using the ping-pong Lemma on some ideal completion (\eg the $p$-Royden boundary, see Corollary \ref{tcorroyd-c}).

\begin{ques}
If $G$ is a finitely generated solvable group, does $\srl{H}^1(G,\lmd_{\ell^pG}) = 0$ for any $1<p<\infty$?
\end{ques}
Already the metabelian (derived length 2) case is not clear. 
Some $2$-generator metabelian groups of these are known to have malnormal subgroups (see de la Harpe \& Weber \cite[\S{}3]{HW}), but the from the possible tools to conclude the vanishing there is always [at least] one which applies.
The case (locally nilpotent not finitely generated)-by-Abelian would probably suffice to answer the question. 
In fact, for such groups, the difficulty comes in when there is a uniform growth bound on the locally nilpotent group (\eg it is locally finite or $\oplus_{i=1}^n \zz[\tfrac{1}{p_i}]$ where $p_i \in \nn$), otherwise Corollary \ref{tasclp-c} can be applied. 

The sharpness of Corollary \ref{tcorptlp-c}, Theorem \ref{tqnorlpcoh-t} and Theorem \ref{tqnorcen-t} are not so easy to check.
M. Bourdon told the author that amalgamated products (see \cite{Bourdon-hyp}) give a first family of examples.
Another prevailing source of malnormal groups is hyperbolicity (see \cite[\S{}3, in particular Example 9]{HW} for more examples). 
In fact, $G$ is hyperbolic relative to the family $\{H_\lmd \}_{\lmd \in \Lmd}$, then the $H_{\lmd}$ are malnormal in $G$ (see Osin's book \cite[Corollary 2.37]{Osin-book}; it is a consequence of the ``fine'' property in the sense of Bowditch \cite[Proposition 2.1 and \S{}4]{Bow}). 
M.~Bourdon also indicated to the author that combining [a careful reading of the] construction of Gerasimov \cite{Gera} with a result of Puls \cite[Theorem 1.3]{Puls-floyd} shows that relatively hyperbolic groups have a non-trivial cohomology for some $p$.
Currently, there are no written reference relating this $p$ to the conformal dimension of the Bowditch boundary \cite{Bow}.

D.~Osin pointed out to the author that some acylindrically hyperbolic groups have a trivial $\ell^p$-cohomology for all $p \in [1,\infty[$.
A result of Lohou\'e \cite{Loh} (see also \cite[Theorem 1.2 or Corollary 3.14]{Go}) shows that, for non-amenable groups, $H^1(G,\lmd_{\ell^pG}) \neq 0$ implies the existence of a harmonic function with gradient in $\ell^p$.
Hence, a weaker condition than ``trivial [reduced] $\ell^p$-cohomology'' as $p \to \infty$ is to ask whether there is a harmonic function with gradient in $c_0$.
At the moment, it is unclear whether or not acylindrically hyperbolic groups always have a harmonic function with gradient in $c_0$.
Harmonic function with gradient in $c_0$ are produced from the harmonic cocycles of strongly mixing unitary representations (see \cite[\S{}2.5 and Corollary 2.6]{GJ} or Lemma \ref{tgradcond-l}).
\begin{ques}
Assume a finitely generated group $G$ has a harmonic function with gradient in $c_0$. Is there a strongly mixing linear isometric representation of $G$ on a strictly convex Banach space with non-trivial reduced cohomology?
\end{ques}
Note that the representation may not (in general) be unitary (there are hyperbolic groups with property (T), see \eg \.Zuk \cite[Theorem 3 and 4]{Zuk}).

For $p\in ]1,\infty[$, the triviality of the reduced $\ell^p$-cohomology in groups is monotone (if it non-trivial for $p$ then it is non-trivial for all $q>p$).
The infimal $p$ where the $\ell^p$-cohomology is non-trivial is sometimes denoted $p_c(G)$.
Corollary \ref{tcoefflpharm-c} seems to relate $p_c(G)$ to the quantity $p(G)$ introduced by Shalom in \cite[\S{}1.8 and \S{}4]{Shal2}.
Note that the conditions on the coefficients are not the same and are there differences in reduced/unreduced cohomology, nevertheless this raises the
question: when is there an inequality $p_c(G) \leq p(G)$?
Links between $\ell^p$-cohomology and $p(G)$ are also hinted at in Bourdon, Martin \& Valette \cite{BMV}.

Let $G$ be a torsion-free group, $\pi:G \to \mr{GL}(V)$ a representation with finite stabilisers and $b \in Z^1(G,\pi)$.
If $b(g) =0$ for some $g \in G$, then Lemma \ref{tlempt-l} implies that $b \equiv 0$ on the malnormal hull of $\pgen{g}$ (because $\pgen{g}$ is infinite and contained in $\ker b$).
This looks like a first step to extend a result from Peterson \& Thom \cite[Theorem 4.1]{PT} to other representation.
The crucial point that ``$\forall b \in Z^1(G,\pi), b(g) =0 \implies b(h) =0$'' implies ``$\forall b\in Z^1(G,\pi), b(g) =0 \iff b(h) =0$'' seems out of reach.

The following question is motivated by \cite[\S{}2.5 and Corollary 1.4]{GJ} (see Definition \ref{defcomp} for compression).
\begin{ques}
Assume $b \in Z^1(G,\pi)$ is a harmonic cocycle for a unitary representation $\pi$ and fix a generating set $S$ for $G$. 
Let $b_n$ be the number of elements in a ball of radius $n$ of $\cayl(G,S)$ and $s_n = b_n - b_{n-1}$ (for $n>0$ and $s_0 = 1$).
If $\rho_-(n)$ is the compression function of $b$, is it true that there is a $K>0$ so that 
\[
\rho_-(n) \leq K \sqrt{\sum_{i=0}^{n} \frac{b_i}{s_i}.}
\]
\end{ques}
The main motivation is the following. If true this would mean that
\[
\rho_-(t) \lesssim
\left\{\begin{array}{ll}
n^{(1-\nu)/2}	& \text{if } b_n \approx \mr{exp}(n^\nu) \\
n /\sqrt{\ln n}	& \text{if } b_n \approx n^{\ln n}
\end{array}\right.
\]
where $\rho_-$ is the compression function of a harmonic cocycle. 
Note that this estimate fails for an element of $\srl{B}^1(G,\pi)$. 
Nevertheless, it would be enough to settle \cite[Conjecture 1]{CTV} for discrete amenable groups. 
Actually, thanks to Naor \& Peres \cite[Theorem 1.1]{NP}, the only amenable groups for which the conjecture is open are those with a diffusive behaviour (the expected distance to the identity of a random walk at time $n$ is $\simeq \text{cst} 
\cdot \sqrt{n}$).

Virtual cocycles are very useful in some aspects and it would be nice to be able to use them for a wider range of representations. 
The following question seems like a natural place to start.
\begin{ques}
Assume $G \acts (X,\mu)$ is mildly mixing and let $\pi$ be the associated $\L^p$-representation. 
If $K<G$ and $f \in \D^p(K,\pi)$ is constant on $K$-orbits (\ie associated to the trivial cocycle), 
\begin{itemize}\renewcommand{\itemsep}{-1ex}
\item what are the choices (depending on $g \in G \setminus N_G(K)$) of the constants so that $\pi(g)f \in \D^p(K,\pi)$?
\item if there is a choice of the constants so that $f \in \D^p(G,\pi)$, does it imply that $K$ is not wq-normal in $G$?
\end{itemize}
\end{ques}
For the first question, note that if $g \in N_G(K)$ any choice would work. Any answer for a different mixing condition would be of interest too.

It seems difficult to pass the arguments of \S{}\ref{suntwist} to reduced cohomology. Here is a list of possible improvements.
\begin{ques}\label{qredcohqnorm}
~\\[-1.5em]
\begin{enumerate}\renewcommand{\itemsep}{-1ex} \renewcommand{\labelenumi}{\bf \arabic{enumi}.}
\item If $\pi$ is mildly mixing and there is an infinite finitely generated $H<G$ which is q-normal and $\srl{H}^1(H,\pi_{|H}) =0$,
does $\srl{H}^1(G,\pi) =0$? 

\item If $\pi$ is mildly mixing and there is an infinite $H \lhd G$ with $\srl{H}^1(H,\pi_{|H}) =0$,
does $\srl{H}^1(G,\pi) =0$? 

\item If $\pi$ is mildly mixing and there is a finitely generated subgroup $H<G$ with $Z^{FC}_G(H)$ infinite, 
is $H$ contained in the kernel of the harmonic cocycles?

\item If $\pi$ is weakly mixing and there is an infinite $H \lhd G$ with $G/H$ cyclic and $\srl{H}^1(H,\pi_{|H}) =0$,
does $\srl{H}^1(G,\pi) =0$? 
\end{enumerate}
\end{ques}
Question \ref{qredcohqnorm}.4 has been answered in the negative by Brieussel \& Zheng \cite[Remark 4.6]{BZ}.
In Question \ref{qredcohqnorm}.1-3, it would be reasonable to add hypothesis such as $G = \pgen{H,g_0}$ or $H^1(G,\pi) = \srl{H}^1(G,\pi)$ or strongly mixing.
Note also that mildly mixing is the close to being the most optimistic mixing hypothesis (one could also put the hypothesis on the subgroups, compare with Lemma \ref{tlempopastyle-l}). 
Indeed, $\zz \wr \zz$ has $\zz^d$ (for any $d>1$) as q-normal subgroup. 
All ergodic representations of $\zz^d$ have trivial reduced cohomology (see \eg Theorem \ref{tergliou-t}). 
However, Shalom \cite[Theorem 15 or Theorem 5.4.1]{Shal} showed that $\zz \wr \zz$ does not have property $H_{FD}$, \ie there is a weakly mixing representation of $\zz \wr \zz$ which has non trivial reduced cohomology. 
These representations are however not mildly mixing nor do they have finite stabilisers (they factor through an infinite subgroup, see \cite[Proof of Theorem 5.4.1]{Shal}).

Let us briefly discuss this construction. 
Let $G = \zz \wr H := ( \oplus_{h \in H} \zz ) \rtimes H$ be a wreath product (with $H$ the ``base'' group and $\zz$ the ``lamp state'' group) and write its elements as $(l,h)$ where $l$ is a finitely supported function from $H$ to $\zz$. Let $\pi$ be a representation of $H$ and look at the representation $\pi: G \to \Uni(\Hi)$ be defined by $\pi(l,h) = \pi(h)$. 
For any vector $\xi \in \Hi \setminus \{0\}$, $b_\xi(l,h) = \sum_{g \in H} l(g) \pi(g) \xi$ defines a cocycle. 
It is easy to check that this cocycle is harmonic: for $S_H$ some generating set of $H$,
\[
b_\xi(-\delta_e,e) + b_\xi(\delta_e,e) + \sum_{s \in S_H} b_\xi(0,s) = -\xi+\xi +0 =0.
\]
Hence $b_\xi$ is non-trivial in cohomology (in fact, $b_\xi - b_\eta = b_{\xi-\eta}$ so each $\xi \in \Hi \setminus \{0\}$ gives a different class). 

The interesting point (for the current subject matter) is that $\ker b = H$ is a malnormal subgroup of $G$ (see de la Harpe \& Weber \cite[Proposition 1 and subsequent \P{}]{HW}).
The kernel of $\pi$ (in $G$) contains the infinite normal subgroup $N = \oplus_{h \in H} \zz$, so it does not have finite stabilisers (and is not mildly mixing).
In fact, $\pi_{|N}$ acts by the trivial representation and, hence,  $\srl{H}^1(N', \pi_{|N'}) \neq 0$ for any subgroup $N' < N$.
So this example is not contradictory with positive answers to Question \ref{qredcohqnorm}.

\textsc{\newline Antoine Gournay }
\newline TU Dresden 
\newline 01062 Dresden, Germany

\end{document}